\documentclass{amsart}
\usepackage[margin=1.2in]{geometry}
\usepackage{amssymb}
\usepackage{amscd}
\usepackage[all]{xy}
\usepackage{bbm}
\usepackage{mathrsfs}
\usepackage{enumerate}
\usepackage{tikz}
\usepackage{stmaryrd}
\usepackage{etoolbox}
\usepackage{array}
\usepackage{verbatim}
\usepackage{hyperref}

\newcommand{\Z}{\mathbb{Z}}

\newcommand{\R}{\mathbb{R}}

\newcommand{\eps}{\varepsilon}
\newcommand{\inter}{\operatorname{int}}

\newcommand{\del}{\nabla}

\newcommand{\lap}{\Delta}

\newcommand{\bd}{\partial}

\newcommand{\eval}{\bigg\vert}

\newcommand{\la}{\langle}
\newcommand{\ra}{\rangle}
\newcommand{\diam}{\operatorname{diam}}

\renewcommand{\div}{\operatorname{div}}

\newcommand{\grad}{\del}

\newcommand{\ric}{\operatorname{Ric}}

\theoremstyle{plain}
\newtheorem{theorem}{Theorem}
 
\newtheorem{corollary}[theorem]{Corollary}
\newtheorem{prop}[theorem]{Proposition}
\newtheorem{lem}[theorem]{Lemma}
\newtheorem{conj}[theorem]{Conjecture}

\theoremstyle{definition}
\newtheorem{defn}[theorem]{Definition}
\newtheorem{rem}[theorem]{Remark}

\begin{document}
\title{On the topology of manifolds with positive intermediate curvature} 
\author{Liam Mazurowski}
\address{Department of Mathematics, Lehigh University, Bethlehem, Pennsylvania, 18015, United States}
\email{lim624@lehigh.edu}
\author{Tongrui Wang}
\address{School of Mathematical Sciences, Shanghai Jiao Tong University, Minhang District, Shanghai 200240, China}
\email{wangtongrui@sjtu.edu.cn}
\author{Xuan Yao}
\address{Department of Mathematics, Cornell University, Ithaca, New York, 14853, United States}
\email{xy346@cornell.edu}

\begin{abstract}
    We formulate a  conjecture relating the topology of a manifold's universal cover with the existence of metrics with positive $m$-intermediate curvature. We prove the result for manifolds of dimension $n\in\{3,4,5\}$  
      and for most choices of $m$ when $n=6$.  As a corollary, we show that a closed, aspherical 6-manifold cannot admit a metric with positive $4$-intermediate curvature.   
\end{abstract}

\maketitle

\section{Introduction}

A manifold $M$ is called aspherical if the universal cover of $M$ is contractible. Equivalently, $M$ is aspherical if $\pi_k(M) = 0$ for $k\ge 2$.  Thus an aspherical manifold $M$ is a $K(\pi,1)$ space for $\pi = \pi_1(M)$.  Schoen and Yau \cite{schoen1987structure} conjectured that a closed aspherical manifold cannot carry a metric of positive scalar curvature; also see Gromov \cite{gromov2006large}.  This $K(\pi,1)$ conjecture is known to be true when $n=3$ by work of Schoen-Yau \cite{schoen1979existence} and Gromov-Lawson \cite{gromov1983positive},  when $n=4$ by Chodosh-Li \cite{chodosh2024generalized}, and when $n=5$ by Chodosh-Li \cite{chodosh2024generalized} and independently by Gromov \cite{gromov2020no}. 

\begin{theorem}
\label{Theorem:Aspherical}
Assume that $M^n$ is a closed $K(\pi,1)$ manifold of dimension $n\in \{3,4,5\}$. Then $M$ does not admit a metric with positive scalar curvature. 
\end{theorem}

Let $T^n$ be the $n$-dimensional torus.  The Geroch conjecture states that $T^n$ does not admit a metric of positive scalar curvature. This can be seen as a special case of the $K(\pi,1)$ conjecture.  The Geroch conjecture was resolved by Schoen-Yau \cite{schoen1979existence} for $n\le 7$ and by Gromov-Lawson \cite{gromov1980spin} in all dimensions.  In a different direction, Brendle-Hirsch-Johne \cite{brendle2024generalization} have generalized the Geroch conjecture to show that a closed manifold $M^n = N^{n-m} \times T^{m}$ does not admit a metric with positive $m$-intermediate curvature when $n\le 7$. Here the $m$-intermediate curvature (see Definition \ref{def: intermediate curvature}) interpolates between the Ricci curvature ($m = 1$) and the scalar curvature ($m = n-1$). 

\begin{theorem}[Brendle-Hirsch-Johne \cite{brendle2024generalization}] 
\label{Theorem:Generalized-Geroch} 
A closed manifold $M^n = N^{n-m} \times T^{m}$ does not admit a metric with positive $m$-intermediate curvature when $n\le 7$. 
\end{theorem}

In this paper, we propose the following conjecture which relates the topology of a manifold's  universal cover with the existence of metrics with positive $m$-intermediate curvature. We note that both Theorem \ref{Theorem:Aspherical} and Theorem \ref{Theorem:Generalized-Geroch} occur as special cases of this conjecture.

\begin{conj}\label{generalized-Kpi1} 
Let $M^n$ be a closed manifold of dimension $3\leq n\leq 7$. Assume that the universal cover $\overline M$ of $M$ satisfies 
\[
H_n(\overline M,\Z) = H_{n-1}(\overline M,\Z) = \hdots = H_{n-m+1}(\overline M,\Z) = 0. 
\]
Then $M$ does not admit a metric with positive $m$-intermediate curvature. 
\end{conj}

As our main result, we show that Conjecture \ref{generalized-Kpi1} is true for $n \in \{3,4,5\}$ and for most choices of $m$ when $n=6$.

\begin{theorem}[Main Theorem]
\label{Theorem:Main-Theorem}
    Suppose $M^n$ is a closed  manifold and that the universal cover $\overline M$ of $M$ satisfies 
\[
H_n(\overline M,\Z) = H_{n-1}(\overline M,\Z) = \hdots = H_{n-m+1}(\overline M,\Z) = 0. 
\]
Then $M$ does not admit a metric of positive $m$-intermediate curvature in any of the following cases: $n\in\{3,4,5\}$ and $m\in \{1,2,\dots,n-1\}$;  $n = 6$ and $m\in\{1,2,3,4\}$; $n \ge 7$ 
and $m = 1$.
\end{theorem}

In dimension $n = 6$, the $5$-intermediate curvature is equivalent to the scalar curvature, and positive $4$-intermediate curvature is a slight strengthening of positive scalar curvature.  As a corollary of the $n= 6$ and $m=4$ case of our main theorem, we see that a closed, aspherical $6$-manifold cannot admit a metric with positive $4$-intermediate curvature.  This provides some evidence for the $K(\pi,1)$ conjecture in dimension 6. 

\begin{corollary}
    A closed, aspherical $6$-manifold does not admit a metric with positive 4-intermediate curvature. 
\end{corollary}

Finally, we note that Chodosh-Li-Liokumovich \cite{chodosh2023classifying} have proven a mapping version of the $K(\pi,1)$ conjecture and a reformulation of the $K(\pi,1)$ conjecture as a classification theorem. We show that corresponding results also hold in our setting. 

\begin{theorem}
Suppose $M^n$ is a closed  manifold and that the universal cover $\overline M$ of $M$ satisfies 
\[
H_n(\overline M,\Z) = H_{n-1}(\overline M,\Z) = \hdots = H_{n-m+1}(\overline M,\Z) = 0. 
\]
Assume further that $N^n$ is a closed manifold with a non-zero degree map to $M^n$.
Then $N$ does not admit a metric of positive $m$-intermediate curvature in any of the following cases: $n\in\{3,4,5\}$ and $m\in \{1,2,\dots,n-1\}$;  $n = 6$ and $m\in\{1,2,3,4\}$; $n\ge 7$ and $m = 1$.
\end{theorem}

\begin{theorem}
    Suppose $M^n$ is a closed manifold which admits positive $m$-intermediate curvature. Suppose further that  
    \begin{enumerate}
        \item $n=5$, $m=2$, and $\pi_2(M^5)=0$; or
        \item $n=6$, $m=4$, and $\pi_2(M^6)=\pi_3(M^6)=\pi_4(M^6)=0$.
    \end{enumerate}
    Then a finite covering of $N$ is homeomorphic to $S^n$ or connected sums of $S^{n-1}\times S^1$.
 \end{theorem}

\subsection{Further Discussion and Proof Ideas}

Schoen and Yau's proof \cite{schoen1979existence} of the Geroch conjecture uses minimal hypersurfaces and an inductive descent argument. A key step in the argument is to show that if $(M,g)$ has positive scalar curvature and $\Sigma$ is an area minimizing hypersurface in $M$ then $\Sigma$ admits a conformal metric with positive scalar curvature. In fact, the necessary conformal factor is a suitable power of the first eigenfunction for the second variation of area. 

Brendle, Hirsch, and Johne \cite{brendle2024generalization} proved a generalization of the Geroch conjecture to closed manifolds of the form $M^n = N^{n-m} \times T^m$.  To do so, they introduced a new curvature condition called $m$-intermediate curvature.

\begin{defn}\label{def: intermediate curvature}
Suppose $(M^n,g)$ is Riemannian manifold.  Given a collection of orthonormal vectors $\{e_1,\cdots,e_m\}$ at a point $p\in M$, let $\{e_1,\cdots e_m,e_{m+1},\hdots, e_n\}$ be an extension to an orthonormal basis of $T_pM$. The {\em $m$-intermediate curvature} $C_m$ of the orthonormal vectors $\{e_1,\cdots,e_m\}$ is defined by 
\[
C_m(e_1,\cdots,e_m):=\sum_{p=1}^m\sum_{q=m+1}^n R_M(e_p,e_q,e_p,e_q).
\]
We say that $(M,g)$ has {\em positive $m$-intermediate curvature} if for any choice of orthonormal vectors $\{e_1,\cdots,e_m\}$ at any point $p\in M$, we have $C_m(e_1,\cdots,e_m)>0$. 
\end{defn}

We note that the 1-intermediate curvature is precisely the Ricci curvature and that the $(n-1)$-intermediate curvature is equal to the scalar curvature up to a constant factor.  The 2-intermediate curvature had previously been introduced and studied by Shen and Ye \cite{shen1996stable} as {\it bi-Ricci} curvature. 

Brendle, Hirsch, and Johne's main result says that $M^n = N^{n-m}\times T^m$ cannot admit a metric with positive $m$-intermediate curvature when $n\le 7$. Their proof is based on the fact that certain weighted minimal slicings do not exist in manifolds with positive $m$-intermediate curvature. In the construction of their weighted slicings, each successive weight $\rho_{k+1}$ is obtained from the previous weight $\rho_k$ by multiplying by the first eigenfunction for the second variation of $\rho_k$-weighted area. This is reminiscent of the Schoen-Yau descent argument. In dimension $n>7$, Xu \cite{xu2024dimension} constructed interesting counter-examples to Brendle-Hirsch-Johne's generalization of the Geroch Conjecture, which is the reason why we only expect Conjecture \ref{generalized-Kpi1} to be true for $3\leq n\leq 7$.

Recall that the Geroch conjecture is a particular case of the $K(\pi,1)$ conjecture. 
In their 1987 paper, Schoen and Yau \cite{schoen1987structure} put forth an outline for how to prove the $K(\pi,1)$ conjecture in dimension $n=4$. Chodosh and Li's proof \cite{chodosh2024generalized} is inspired by this outline.  Let $M^n$ be a closed, aspherical manifold.  Assume for contradiction that $M$ has a metric with positive scalar curvature. Working in the universal cover $\overline M$ of $M$, Chodosh and Li construct a geodesic line $\sigma$ and a closed null-homologous submanifold $\Lambda^{n-2}$ which is far away from $\sigma$ but linked with $\sigma$.   Let $\Sigma_1$ be the area minimizing filling of $\Lambda$.  Using a clever choice of prescribed mean curvature functional, Chodosh and Li construct a $\mu$-bubble $\Sigma_2\subset \Sigma_1$ such that $\Lambda$ is homologous to $\Sigma_2$ in $\Sigma_1$ and $\Sigma_2$ lies far away from $\sigma$. When $n= 4$, the $\mu$-bubble $\Sigma_2$ is 2-dimensional and it is possible to show with a second variation argument that each component of $\Sigma_2$ has uniformly bounded diameter.  Thus, by a general property of universal covers, it is possible to fill each component of $\Sigma_2$ within a fixed sized neighborhood in $\overline M$. In this manner, one obtains a new filling of $\Lambda$ which does not intersect $\sigma$ and this is a contradiction.

When $n=5$, new difficulties arise because the $\mu$-bubble $\Sigma_2$ is 3-dimensional and may not have a diameter bound. To overcome this, Chodosh and Li developed the so-called slice-and-dice argument.  They first slice $\Sigma_2$ open along suitable $\mu$-bubbles $\Sigma_{3,1},\hdots,\Sigma_{3,k}$ to obtain a manifold $\hat \Sigma_2$ with simple 2 dimensional homology.  Then they dice $\hat \Sigma_2$ into pieces with small diameter using free boundary $\mu$-bubbles.  These free boundary $\mu$-bubbles are two dimensional, so it is possible to use a second variation argument together with the Gauss-Bonnet theorem to show that each dicing surface is a disk. Crucially, the boundary of a disk is connected, and this gives a simple combinatorial structure to the output of the slice and dice procedure. This can then be used to construct a filling of $\Sigma_2$ within a fixed sized neighborhood in $\overline M$ as before. 

Our proof of Theorem \ref{Theorem:Main-Theorem} is based on the argument of Chodosh and Li. The main observation is that positive $m$-intermediate curvature should imply a diameter bound on the slice $\Sigma_{m-1}$ in low dimensions.   A key tool in this regard is the generalized Bonnet-Myers theorem discovered by Shen and Ye \cite{shen1997geometry}. We will apply the generalized Bonnet-Myers theorem with a weight equal to a product of eigenfunctions for the stability operators of certain weighted area and $\mu$-bubble functionals. We remark that the case $n=6$ and $m=3$ is very delicate and requires the full strength of the Shen-Ye generalized Bonnet-Myers theorem as well as the use of a carefully chosen weighted $\mu$-bubble functional.

To handle the case $n=6$ and $m=4$, we adapt the slice and dice procedure of Chodosh and Li.  Since the dicing submanifolds are now 3-dimensional, we can no longer use the Gauss-Bonnet theorem to determine their topology.  Nevertheless, we can prove the following Frankel type theorem which implies that the dicing submanifolds have connected boundary: 

\begin{theorem}
    Let $\Sigma$ be a $3$-dimensional compact Riemannian manifold with boundary. Assume $\Sigma$ admits a positive function $f$ which satisfies 
    \[
    \ric(v,v) - f^{-1}\lap f + \frac{1}{2} \vert \grad \ln f\vert^2 \ge 0
    \]
    for all unit vectors $v$. Moreover, suppose that $\bd_\eta f = -f H_{\bd \Sigma}$ where $\eta$ is the unit outward normal to $\bd \Sigma$. Then $\bd \Sigma$ is connected. 
\end{theorem}

The fact that the dicing submanifolds have connected boundary is then enough to carry through the remainder of the argument. 

\begin{rem} The hypothesis of Shen and Ye's theorem is dimension dependent, and we were not able to verify the hypothesis on the slice $\Sigma_{m-1}$ when $n=7$ and $m\in\{2,3,4\}$. 
\end{rem}

\subsection{Organization}
In Section \ref{sec: pre}, we first introduce some convenient terminology and then we recall basic facts about $\mu$-bubbles, including general existence and stability results. We prove a Frankel type theorem for positive conformal Ricci curvature in Section \ref{sec: Frankel}. Combining the weighted slicing techniques, Shen-Ye's diameter estimate, and the slice-and-dice procedure of Chodosh-Li, we prove the Main Theorem in Section \ref{sec: proof of main theorem}. In Section \ref{sec: mapping and classification}, we generalize our results to a mapping version and prove a refinement of the Main Theorem into a positive result.

\subsection{Acknowledgments}
We would like to thank Xin Zhou for his helpful discussions and guidance on this project. L.M. acknowledges the support of an AMS-Simons travel grant. 

\section{Preliminaries}\label{sec: pre}

In this section, we discuss some preliminary results that will be needed in the proof of our main theorem. First, we introduce some convenient terminology for manifolds where the homology of the universal cover vanishes in certain degrees.

\begin{defn}
    We say a closed manifold $M^n$ is {\em $m$-acyclic}, $m\in\{1,2\cdots,n-1\}$, if the universal cover $\overline M$ of $M$ satisfies 
    \[
    H_{n-m+1}(\overline{M};\mathbb Z)=\cdots=H_n(\overline{M};\mathbb Z)=0.
    \]
\end{defn}

\subsection{Topological Preliminaries}

In this subsection, we discuss some topological results about the universal cover of a closed manifold. 

\begin{lem}
    Suppose $(M^n,g)$ is an $m$-acyclic closed manifold for some $m\in\{1,2,\cdots,n-1\}$. Then there exists a geodesic line on $(\overline{M},\bar{g})$, where $\bar{g}$ is the lifted Riemannian metric.
\end{lem}
\begin{proof}
    Since $H_n(\overline{M};\mathbb Z)=0$, $\overline{M}$ is non-compact. The proof follows easily as in \cite[Lemma 6]{chodosh2024generalized}.
\end{proof}

More generally, we recall the following construction of Chodosh and Li \cite{chodosh2024generalized}. 

\begin{prop}
\label{linking} 
Assume that $\overline M$ is the universal cover of a closed Riemannian manifold $M^n$.  Further suppose that $H_{n}(\overline M,\Z) = H_{n-1}(\overline M,\Z) = 0$.  Then for any $L > 0$, there exists a geodesic line $\sigma$ in $\overline M$ together with a closed, null-homologous manifold $\Lambda^{n-2}$ embedded in $\overline M$ such that 
\begin{itemize}
\item[(i)] $\Lambda$ is linked with $\sigma$, i.e., if $\Lambda = \bd \Sigma$ then $\Sigma$ has non-zero algebraic intersection number with $\sigma$,
\item[(ii)] $d(\Lambda,\sigma) \ge L$. 
\end{itemize}
\end{prop}

\begin{proof}
This follows from the argument in \cite[Section 2]{chodosh2024generalized}. While Chodosh and Li consider the universal cover of a closed aspherical manifold, it is easy to see that their arguments apply to the universal cover $\overline M$ of an arbitrary closed manifold $M^n$ provided that $H_{n}(\overline M,\Z) = H_{n-1}(\overline M,\Z) = 0$. 
\end{proof}

\begin{prop}
\label{Proposition:uniformly-acyclic}
    Suppose $M^n$ is $m$-acyclic. 
    For any $r>0$, there exists $R(r)$ with the following property. For any $k\geq n-m+1$ and any $k$-cycle $\alpha$  with $\alpha\subset B_r(p)$, we have $\alpha=\partial\beta$ with $\beta\subset B_R(p)$.
\end{prop}
\begin{proof}
    The proof is exactly the same as in Chodosh-Li \cite{chodosh2024generalized}; see also He-Zhu \cite{he2024noterationalhomologyvanishing}. 
\end{proof}

\subsection{Preliminaries on $\mu$-bubbles}
Next we recall some basic facts about $\mu$-bubbles, including general existence and stability results.

For $n\leq 7$, consider a Riemannian manifold $(M^n,g)$ with boundary $\partial M=\partial_-M\sqcup \partial_+ M$, where neither of $\partial_{\pm}M$ is empty. 
Suppose $u$ is a smooth positive function on $M$, and $h$ is a smooth function defined on the interior of $M$ so that $h\to \pm\infty$ as $x\to \partial_{\pm}M$ respectively. 
Given a Caccioppoli set $\Omega_0\subset M$ with smooth boundary and containing  $\partial_{+}M$, consider the $\mu$-bubble functional:
\begin{align}
    \mathcal A(\Omega):=\int_{\partial^*\Omega}uda-\int_{M}(\chi_{\Omega}-\chi_{\Omega_0})hudv,
\end{align}
for all Caccioppoli set $\Omega\subset M$ with  $\Omega\Delta\Omega_0\subset\subset \inter(M)$, where $\bd^*\Omega\subset \inter (M)$ is the reduced boundary of $\Omega$ in $\inter (M)$. We call the $\mathcal A$-minimizer $\Omega$ in this class a {\em $\mu$-bubble}.

The existence and regularity of a minimizer of $\mathcal A$ among all Caccioppoli sets was claimed by Gromov \cite{gromov2019four} and proven rigorously by Zhu \cite{zhu2021width}. 

\begin{prop}[Gromov\cite{gromov2019four}, Zhu \cite{zhu2021width}]\label{prop: mu-bubble regularity}
There exists a smooth minimizer $\Omega$ for $\mathcal A$ such that $\Omega\Delta\Omega_0\subset\subset\inter(M^n)$, for $3\leq n\leq 7$. 
\end{prop}

In Chodosh-Li's proof \cite{chodosh2024generalized} of the $K(\pi,1)$ conjecture in dimension $4$ and $5$, they generalized the $\mu$-bubble techniques to a free boundary version. In their setting, they assume that $(M^n,g)$ is a Riemannian manifold with co-dimension $2$ corners in the sense that any boundary point has a neighborhood diffeomorphic to one of the following: $\{x\in\mathbb R^n: x_n\geq 0\}$ or $\{x\in\mathbb R^n: x_{n-1},x_n\geq 0\}$. Furthermore, $\partial M=\partial_{\pm}M\cup \partial_0M$, $\partial_{\pm}M$ meets $\partial_0M$ orthogonally, and $\partial_{\pm}M\cap\partial_0M$ consists of smooth co-dimension $2$ closed submanifolds. In the free boundary setting, we assume that 
\[
H_{\partial_0 M}+u^{-1}\langle\nabla_Mu,\nu_{\partial_0M}\rangle=0.
\]
For the functional $\mathcal A$ defined as before, Chodosh-Li used similar arguments as in Zhu's work \cite{zhu2021width} and proved the following results.

\begin{prop}[Chodosh-Li \cite{chodosh2024generalized}]
\label{proposition:variational-formulas}
    Suppose $n\leq 7$, there exists $\Omega$ with $\partial\Omega\subset \inter(M)\cup\partial_0M$ minimizing $\mathcal A$ among such regions. The boundary $\partial\Omega$ is smooth and meets $\partial_0M$ orthogonally. We have
    \[
    H=-u^{-1}\langle\nabla_Mu,\nu_{\partial\Omega}\rangle+h
    \]
    along $\partial \Omega$, where $H$ is the mean curvature of $\bd\Omega$ with respect to the unit outer normal $\nu_{\bd\Omega}$. Finally, if $\Sigma$ is a component of $\partial\Omega$, then for any $\psi\in C^1(\Sigma)$, we have
    \begin{align*}
        0 &\leq \left.\frac{d^2}{dt^2}\right|_{t=0}\mathcal{A}(\Omega^t)
        \\ &=\int_{\Sigma} -u\psi\Delta_{\Sigma}\psi-\left(\|A_{\Sigma}\|^2+\ric_{M}(\nu,\nu)\right) u \psi^2-\psi\langle\nabla_\Sigma u,\nabla_{\Sigma}\psi\rangle 
        \\ &\qquad\qquad+ \psi^2\nabla^2_M u(\nu,\nu) -\psi^2u^{-1}\langle\nabla_{M}u,\nu\rangle^2-u\psi^2\langle\nabla_Mh,\nu\rangle ~d\mathcal{H}^{n-1}\\
        &\qquad\qquad +\int_{\partial\Sigma} u\psi \frac{\bd\psi}{\bd\eta} -A_{\partial_0M}(\nu_{\partial\Sigma},\nu_{\partial\Sigma}) u \psi^2~d\mathcal{H}^{n-2},
    \end{align*}
    where $\nu=\nu_{\bd\Omega}$ is the unit outer normal of $\Omega$, $\eta$ is the unit co-normal of $\Sigma$ along $\bd\Sigma$, and $\{\Omega^t\}$ is a smooth family of regions with $\Omega^0=\Omega$ and normal speed $\psi$ at $t=0$.
\end{prop}

\subsection{Diameter Estimates} 

\label{Section:diameter-bound}

The classical Bonnet-Myers theorem gives a diameter bound for manifolds with positive Ricci curvature. Shen and Ye discovered the following diameter bound for stable minimal hypersurfaces in manifolds with positive 2-intermediate curvature. 

\begin{theorem}[Shen-Ye \cite{shen1996stable}]
\label{Theorem:Shen-Ye-BiRicci}
Let $M^n$ be a complete Riemannian manifold of dimension $n\in \{3,4,5\}$. Assume that the $2$-intermediate curvature of $M$ is at least $\kappa > 0$. Let $\Sigma$ be a stable minimal hypersurface in $M$. Then for every $p\in \Sigma$, one has  
\[
d(p,\bd \Sigma) \le \sqrt{c(n)}\frac{\pi}{\sqrt \kappa},
\]
where $c(n)$ is a dimensional constant. 
\end{theorem}

Shen and Ye also proved the following generalized Bonnet-Myers theorem which will be very useful for proving diameter bounds for certain weighted minimal slicings. We state the 3-dimensional version first since it is simpler and more powerful. 

\begin{theorem}[Shen-Ye \cite{shen1997geometry}]
\label{Theorem:generalized-Bonnet-Myers}
    Let $N^3$ be a complete Riemannian manifold. Assume there is a function $f > 0$ on $N$ such that 
    \begin{equation}
    \label{conf-ricci}
    \ric_N(e,e) - f^{-1}\lap_N f + \frac 1 2  \vert \grad_N \ln f\vert^2 \ge \kappa > 0
    \end{equation}
    for all unit tangent vectors $e$. Then 
    \[
    \diam(N) \le \sqrt 2\frac{\pi }{\sqrt \kappa}. 
    \]
    In particular, $N$ is compact. 
\end{theorem}

\begin{defn}
Following Shen-Ye, we will say a manifold satisfying (\ref{conf-ricci}) admits positive {\it conformal Ricci curvature}.
\end{defn}

In higher dimensions, the theorem must be modified as follows. 

\begin{theorem}[Shen-Ye \cite{shen1997geometry}]
\label{Shen-Ye-Higher-Dimension} Let $N^k$ be a complete Riemannian manifold of dimension $k\ge 4$. Assume there is a function $f > 0$ on $N$ and $\tau,\eps > 0$ so that 
\[
\ric_N(e,e) - \tau f^{-1}\lap_N f + \left[\tau - \left(\frac{k-1}{4} +\eps\right) \tau^2 \right] \vert \grad_N \ln f\vert^2 \ge \kappa > 0
\]
for all unit tangent vectors $e$. 
Then the diameter of $N$ is bounded above in terms of $k$, $\kappa$, and $\eps$. 
\end{theorem}

\section{Frankel Type Theorems} \label{sec: Frankel}

Recall that Frankel's theorem  says that if $M$ has positive Ricci curvature and minimal boundary then $\bd M$ is connected.  In this section, we are going to prove a Frankel type theorem for manifolds with positive conformal Ricci curvature.  
As motivation, we first prove the following version of Frankel's theorem for bi-Ricci curvature. This theorem fits into the general theme, first observed by Shen-Ye \cite{shen1996stable}, that stable minimal hypersurfaces in manifolds with positive bi-Ricci curvature behave as if they had positive Ricci curvature. 
Curiously, in contrast to the diameter bounds, the following theorem does not require any dimension restriction.

\begin{theorem}
    Let $M^{n+1}$ be a compact manifold with boundary. Assume that $M$ has positive bi-Ricci curvature and minimal boundary. Assume that $\Sigma^n$ is a two-sided, stable, free boundary minimal hypersurface in $M$. Then $\bd \Sigma$ is connected. 
\end{theorem}

\begin{proof}
    
    Let $\nu$ be the unit normal vector to $\Sigma$ in $M$ and let $\eta$ be the unit outward co-normal along $\bd \Sigma$. The second variation formula says that for any smooth function $\psi$ on $\Sigma$ we have  
    \begin{align*}
    0 &\le \int_\Sigma \vert \grad_\Sigma \psi^2\vert - (\vert A_\Sigma\vert^2 + \ric(\nu,\nu))\psi^2 \, dv - \int_{\bd \Sigma} A_{\bd M}(\nu,\nu)\psi^2\, da\\
    &= -\int_\Sigma \psi J_\Sigma \psi \, dv + \int_{\bd \Sigma} \psi \frac{\bd \psi}{\bd \eta} - A_{\bd M}(\nu,\nu)\psi^2\, da,
    \end{align*}
    where $J_\Sigma = \lap_\Sigma + \vert A_\Sigma\vert^2 + \ric(\nu,\nu)$ is the Jacobi operator. Let $f > 0$ be a first eigenfunction so that 
    \[
    \begin{cases}
        J_\Sigma f + \lambda f = 0,\\
        \bd_\eta f = A_{\bd M}(\nu,\nu) f,
    \end{cases}
    \]
    with $\lambda \ge 0$. 
    We are going to consider a weighted length functional with weight $f$ as in Shen-Ye \cite{shen1996stable}. 

    More precisely, given a curve $c$ in $\Sigma$, define 
    \[
    L(c) = \int_0^\ell f \vert \dot c\vert\, dt. 
    \]
    We now restrict the calculations in $\Sigma$ and follow \cite{shen1996stable} to compute the first and second variation of $L$. Assume that $c(t)$ is a unit speed curve in $\Sigma$. Let $c(s,t)$ be a variation with $c(0,t) = c(t)$. We compute 
    \begin{align*}
        \frac{\bd L}{\bd s} = \int_0^\ell \frac{\bd}{\bd s} (f \langle \dot c,\dot c\rangle^{1/2})\, dt = \int_0^\ell \langle \grad f, \frac{\bd c}{\bd s}\rangle \vert \dot c\vert + f \frac{\langle \nabla_{\bd/\bd s}\dot c, \dot c\rangle}{\vert \dot c\vert}\, dt. 
    \end{align*}
    Therefore setting $V(t) = \frac{\bd c}{\bd s}(0,t)$ we get the first variation formula 
    \[
    \frac{\bd L}{\bd s}\eval_{s=0} = \int_0^\ell \la \grad f,V\ra + f\la \del_{\dot c} V, \dot c\ra \, dt. 
    \]
    
    Assuming $c$ is a critical point among curves with endpoints constrained to lie in $\bd \Sigma$, we can test this against $V$ with $V(0)\in T_{c(0)}\bd \Sigma$ and $V(\ell)\in T_{c(\ell)}\bd \Sigma$ to get 
    \begin{align*}
        0 &= \int_0^\ell \la \grad f,V\ra + f \frac{\bd}{\bd t}\la V,\dot c\ra - f\la V,\del_{\dot c} \dot c\ra\, dt\\
        &= \int_0^\ell \la \grad f,V\ra + \frac{d}{dt}\left[f \la V,\dot c\ra \right] - \la \grad f,\dot c\ra \la V,\dot c\ra - f\la V,\del_{\dot c}\dot c\ra \,dt \\
        &= f\la V,\dot c\ra \eval_{0}^\ell + \int_0^\ell \la \grad f, V\ra - \la \grad f,\dot c\ra \la V,\dot c\ra - f\la V,\del_{\dot c}\dot c\ra \,dt\\
        &= f\la V,\dot c\ra \eval_0^\ell + \int_0^\ell \la (\grad f)^\perp, V\ra - f\la V,\del_{\dot c}\dot c\ra \,dt. 
    \end{align*}
    Here $(\grad f)^\perp$ is the part of $\grad f$ orthogonal to $\dot c$. 
    Therefore $c$ satisfies the weighted geodesic equation 
    \[
    \del_{\dot c}{\dot c} = f^{-1} (\grad f)^\perp,
    \]
    and $c$ meets $\bd \Sigma$ orthogonally at both endpoints. 
    Next we compute the 2nd variation to get 
    \begin{align*}
        \frac{\bd^2 L}{\bd s^2}\eval_{s=0} &= \int_0^\ell \langle \nabla_{\frac{\bd c}{\bd s}} \grad f, V\rangle + \la \grad f , \del_{\frac{\bd c}{\bd s}} \frac{\bd c}{\bd s}\ra + 2\la \grad f,V\ra \la \del_{\frac{\bd c}{\bd s}} \frac{\bd c}{\bd t}, \frac{\bd c}{\bd t}\ra + f \frac{\bd}{\bd s}\la \del_{\frac{\bd c}{\bd s}} \frac{\bd c}{\bd t}, \frac{\bd c}{\bd t}\ra - f \la \del_{\frac{\bd c}{\bd s}} \frac{\bd c}{\bd t},\frac{\bd c}{\bd t}\ra^2\, dt \\
        &= \int_0^\ell \operatorname{Hess}^\Sigma_f(V,V) + \la \grad f, \del_{\frac{\bd c}{\bd s}} \frac{\bd c}{\bd s}\ra + 2 \la \grad f,V\ra \la \del_{\frac{\bd c}{\bd t}} \frac{\bd c}{\bd s},\frac{\bd c}{\bd t}\ra + f\frac{\bd}{\bd s}\la \del_{\frac{\bd c}{\bd t}} \frac{\bd c}{\bd s}, \frac{\bd c}{\bd t} \ra - f \la \del_{\frac{\bd c}{\bd t}} \frac{\bd c}{\bd s},\frac{\bd c}{\bd t} \ra^2\, dt\\ 
      &=\int_0^\ell  \operatorname{Hess}^\Sigma_f(V,V) + \la \grad f, \del_{\frac{\bd c}{\bd s}} \frac{\bd c}{\bd s}\ra + f \la \del_{\frac{\bd c}{\bd s}} \del_{\frac{\bd c}{\bd t}} \frac{\bd c}{\bd s},\frac{\bd c}{\bd t} \ra \, + f \la \del_{\frac{\bd c}{\bd t}}\frac{\bd c}{\bd s},\del_{\frac{\bd c}{\bd s}} \frac{\bd c}{\bd t}\ra\, dt \\
      &\qquad\quad +\int_0^\ell  2 \la \grad f, V\ra \la \del_{\frac{\bd c}{\bd t}}\frac{\bd c}{\bd s},\frac{\bd c}{\bd t}\ra - f\la \del_{\frac{\bd c}{\bd t}}\frac{\bd c}{\bd s},\frac{\bd c}{\bd t}\ra^2\, dt \\
        &=\int_0^\ell  \operatorname{Hess}^\Sigma_f(V,V) + \la \grad f, \del_{\frac{\bd c}{\bd s}} \frac{\bd c}{\bd s} \ra + f \la \del_{\frac{\bd c}{\bd t}}\del_{\frac{\bd c}{\bd s}}\frac{\bd c}{\bd s},\frac{\bd c}{\bd t}\ra - f R^\Sigma(\frac{\bd c}{\bd s},\frac{\bd c}{\bd t},\frac{\bd c}{\bd s},\frac{\bd c}{\bd t})\, dt \\
        &\qquad\quad +\int_0^\ell f \la \del_{\frac{\bd c}{\bd t}}\frac{\bd c}{\bd s} , \del_{\frac{\bd c}{\bd t}}\frac{\bd c}{\bd s}\ra + 2 \la \grad f, V\ra \la \del_{\frac{\bd c}{\bd t}}\frac{\bd c}{\bd s},\frac{\bd c}{\bd t}\ra  - f\la \del_{\frac{\bd c}{\bd t}}\frac{\bd c}{\bd s},\frac{\bd c}{\bd t}\ra^2 \, dt\\
        &=\int_0^\ell  \operatorname{Hess}^\Sigma_f(V,V) + \la \grad f, \del_{\frac{\bd c}{\bd s}} \frac{\bd c}{\bd s} \ra + f \frac{\bd}{\bd t} \la \del_{\frac{\bd c}{\bd s}} \frac{\bd c}{\bd s},\frac{\bd c}{\bd t}\ra - f\la \del_{\frac{\bd c}{\bd s}} \frac{\bd c}{\bd s}, \del_{\frac{\bd c}{\bd t}} \frac{\bd c}{\bd t} \ra  - fR^\Sigma(\frac{\bd c}{\bd s},\frac{\bd c}{\bd t},\frac{\bd c}{\bd s},\frac{\bd c}{\bd t}) \, dt\\
        &\qquad\quad +\int_0^\ell f \la \del_{\frac{\bd c}{\bd t}}\frac{\bd c}{\bd s} , \del_{\frac{\bd c}{\bd t}}\frac{\bd c}{\bd s} \ra + 2 \la \grad f, V\ra \la \del_{\frac{\bd c}{\bd t}}\frac{\bd c}{\bd s},\frac{\bd c}{\bd t} \ra  - f\la \del_{\frac{\bd c}{\bd t}}\frac{\bd c}{\bd s},\frac{\bd c}{\bd t} \ra^2 \, dt.
    \end{align*}
    If $c$ minimizes $L$ then this will be non-negative for all admissible  variations.  
    
    Now assume for contradiction that $\Sigma$ has two distinct boundary components $\Sigma_1$ and $\Sigma_2$. Let $c$ be a unit speed curve which minimizes $L$ over all curves connecting $\Sigma_1$ to $\Sigma_2$. We select $e_1 = \dot c$, $e_2$, $\hdots$, $e_n$  to be an orthonormal frame along $c$. Then
    \[
    \la \del_{\dot c} e_j,\dot c\ra =  -\la e_j,\del_{\dot c} \dot c \ra = -\la e_j, f^{-1}\grad f\ra 
    \]
    for $j= 2,\hdots,n$. Actually, we can further select $e_2,\hdots,e_n$ to be parallel in the normal bundle of $c$ so that $\del_{\dot c}e_j = -\la e_j,f^{-1}\grad f \ra \dot c$. 
   
     We select variations $c_j(s,t)$ with $V_j = e_j$ and $c_j(s,0) \in \Sigma_1$ and $c_j(s,\ell) \in \Sigma_2$ for $j=2,\hdots,n$. Plugging these into the second variation formula and summing over $j$ we get 
    \begin{align*} 
    0 &\le \int_0^\ell \lap_\Sigma f - \operatorname{Hess}^\Sigma_f(\dot c,\dot c) + \sum_{j=2}^n \la \grad f,\del_{\frac{\bd c_j}{\bd s}} \frac{\bd c_j}{\bd s}\ra + f\sum_{j=2}^n \frac{\bd}{\bd t} \la \del_{\frac{\bd c_j}{\bd s}} \frac{\bd c_j}{\bd s}, \dot c\ra - f\ric^\Sigma(\dot c,\dot c)\, dt \\
    &\qquad \quad  + \int_0^\ell - f\sum_{j=2}^n \la \del_{\frac{\bd c_j}{\bd s}} \frac{\bd c_j}{\bd s}, \del_{\dot c} \dot c \ra + 2\sum_{j=2}^n \la \grad f, e_j\ra \la \del_{\dot c} e_j, \dot c\ra - f\sum_{j=2}^n \la \del_{\dot c} e_j, \dot c\ra^2 + f \sum_{j=2}^n  \la \del_{\dot c} e_j,\del_{\dot c} e_j\ra\, dt\\
    &= \int_0^\ell \lap_\Sigma f - \operatorname{Hess}^\Sigma_f(\dot c,\dot c) + \sum_{j=2}^n \la \grad f,\del_{\frac{\bd c_j}{\bd s}} \frac{\bd c_j}{\bd s}\ra + \sum_{j=2}^n \frac{\bd}{\bd t}\left[f \la \del_{\frac{\bd c_j}{\bd s}} \frac{\bd c_j}{\bd s}, \dot c\ra\right] - \sum_{j=2}^n \la \grad f,\dot c\ra \la \del_{\frac{\bd c_j}{\bd s}} \frac{\bd c_j}{\bd s}, \dot c\ra\, dt\\
    &\qquad \quad + \int_0^\ell - f\sum_{j=2}^n \la \del_{\frac{\bd c_j}{\bd s}} \frac{\bd c_j}{\bd s}, f^{-1} (\grad f)^\perp \ra -2\sum_{j=2}^n \la \grad f, e_j \ra \la e_j,f^{-1}\grad f\ra - f\sum_{j=2}^n \la e_j,f^{-1}\grad f\ra^2\, dt  \\
    & \qquad \quad + \int_0^\ell - f\ric^\Sigma(\dot c,\dot c) + \sum_{j=2}^n f^{-1} \la \grad f,e_j\ra^2 \, dt\\
    &= \int_0^\ell \lap_\Sigma f - \operatorname{Hess}^\Sigma_f(\dot c,\dot c)  + \sum_{j=2}^n \frac{\bd}{\bd t}\left[f \la \del_{\frac{\bd c_j}{\bd s}} \frac{\bd c_j}{\bd s}, \dot c\ra\right]  -2 f^{-1}\vert (\grad f)^\perp\vert^2   - f\ric^\Sigma(\dot c,\dot c)\, dt.
    \end{align*}
    
    Now observe that 
    \[
    \operatorname{Hess}_f^\Sigma(\dot c,\dot c) = \la \del_{\dot c} \grad f, \dot c\ra = \frac{d^2 f}{dt^2} - \la \grad f,\del_{\dot c}\dot c\ra = \frac{d^2f}{dt^2} - f^{-1} \vert (\grad f)^\perp\vert^2.
    \]
    Hence we obtain 
    \begin{align*}
        0 &\le \int_0^\ell -f\vert  A_\Sigma\vert^2 - f \ric(\nu,\nu) -f\ric^\Sigma(\dot c,\dot c) - \frac{d^2 f}{dt^2} + \sum_{j=2}^n \frac{\bd}{\bd t} \left[f \la \del_{\frac{\bd c_j}{\bd s}} \frac{\bd c_j}{\bd s},\dot c\ra \right] - f^{-1}\vert (\grad f)^\perp\vert^2\, dt  \\
        &= -\int_0^\ell f\left[\ric(\nu,\nu) + \ric(\dot c,\dot c) - R(\nu,\dot c,\nu,\dot c)\right]\, dt - \int_0^\ell \frac{d^2 f}{dt^2}\, dt - \int_0^\ell f^{-1} \vert (\grad f)^\perp\vert^2\, dt \\
        & \qquad \quad -\int_0^\ell f \left[\vert A_\Sigma\vert^2 + \sum_{j=2}^n (A_\Sigma(\dot c,\dot c) A_\Sigma(\dot e_j,e_j) - A_\Sigma(\dot c,e_j)^2)\right]\, dt\\
        & \qquad \quad - f(c(\ell)) H_{\Sigma_2}(c(\ell)) - f(c(0))H_{\Sigma_1}(c(0)) \phantom{\int}\\
        &< -\bd_\eta f(c(\ell)) - \bd_\eta f(c(0)) - f(c(\ell)) H_{\Sigma_2}(c(\ell)) - f(c(0))H_{\Sigma_1}(c(0)).  \phantom{\int}
    \end{align*}
    Here we used the fact that $\ric(\nu,\nu) + \ric(\dot c,\dot c) - R(\nu,\dot c,\nu,\dot c) > 0$ by the assumption on the bi-Ricci curvature, and the fact that 
    \[
    \vert A_\Sigma\vert^2 + \sum_{j=2}^n (A_\Sigma(\dot c,\dot c) A_\Sigma( e_j,e_j) - A_\Sigma(\dot c,e_j)^2) \ge 0
    \]
    since $\Sigma$ is minimal; see \cite[Equation 15]{shen1996stable}.
    Finally, it remains to note that 
    \[
    \bd_\eta f = A_{\bd M}(\nu,\nu) f = -f H_{\bd \Sigma} 
    \]
    since $\bd M$ is minimal and $\bd \Sigma$ meets $\bd M$ orthogonally.  Thus the final term in the previous chain of inequalities is equal to 0 and we get our contradiction.
\end{proof}

Next, we prove a Frankel type theorem for manifolds with positive conformal Ricci curvature. 

\begin{theorem}
\label{Theorem:ConformalRicciFrankel}
    Let $\Sigma$ be a $3$-dimensional compact Riemannian manifold with boundary. Assume $\Sigma$ admits a positive function $f$ which satisfies 
    \[
    \ric(v,v) - f^{-1}\lap f + \frac{1}{2} \vert \grad \ln f\vert^2 \ge 0
    \]
    for all unit vectors $v$. Moreover, suppose that $\bd_\eta f = -f H_{\bd \Sigma}$ where $\eta$ is the unit outward normal to $\bd \Sigma$. Then $\bd \Sigma$ is connected. 
\end{theorem}

\begin{proof}
    The argument is similar to the previous one but with a slight improvement coming from a modified choice of the variations. Assume for contradiction that $\bd \Sigma$ has two distinct connected components $\Sigma_1$ and $\Sigma_2$.  We let $c$ be a unit speed curve which minimizes the $f$-weighted length from $\Sigma_1$ to $\Sigma_2$. Again we let $e_1 = \dot c,e_2,\hdots,e_n$ be an orthonormal frame along $c$ such that $e_2,\hdots,e_n$ are parallel in the normal bundle of $c$. However, this time we choose variations $c_j(s,t)$ with $V_j = f^{-1/2}e_j$ and $c_j(s,0) \in \Sigma_1$ and $c_j(s,\ell) \in \Sigma_2$ for $j=2,\hdots,n$. 

    Applying the second variation formula to each $c_j$ and then summing over $j$, we deduce that 
    \begin{align*}
        0 &\le \int_0^\ell f^{-1} \lap f - f^{-1}\operatorname{Hess}_f(\dot c,\dot c) + \sum_{j=2}^n \la \grad f,\del_{\frac{\bd c_j}{\bd s}} \frac{\bd c_j}{\bd s}\ra + f\sum_{j=2}^n \frac{\bd}{\bd t} \la \del_{\frac{\bd c_j}{\bd s}} \frac{\bd c_j}{\bd s}, \dot c\ra - \ric(\dot c,\dot c)\, dt \\
    &\qquad \quad  + \int_0^\ell - f\sum_{j=2}^n \la \del_{\frac{\bd c_j}{\bd s}} \frac{\bd c_j}{\bd s}, \del_{\dot c} \dot c \ra + 2\sum_{j=2}^n \la \grad f, f^{-1/2} e_j\ra \la \del_{\dot c} (f^{-1/2}e_j), \dot c\ra - f\sum_{j=2}^n \la \del_{\dot c} (f^{-1/2} e_j), \dot c\ra^2\, dt \\
    &\qquad \quad +\int_0^\ell f \sum_{j=2}^n  \la \del_{\dot c} (f^{-1/2} e_j),\del_{\dot c} (f^{-1/2}e_j)\ra\, dt. 
    \end{align*}
    This simplifies to give 
    \begin{align*}
    0 &\le \int_0^\ell f^{-1} \lap f - f^{-1}\operatorname{Hess}_f(\dot c,\dot c) + \sum_{j=2}^n \la \grad f,\del_{\frac{\bd c_j}{\bd s}} \frac{\bd c_j}{\bd s}\ra + \sum_{j=2}^n \frac{\bd}{\bd t}\left[f \la \del_{\frac{\bd c_j}{\bd s}} \frac{\bd c_j}{\bd s}, \dot c\ra\right] - \sum_{j=2}^n \la \grad f,\dot c\ra \la \del_{\frac{\bd c_j}{\bd s}} \frac{\bd c_j}{\bd s}, \dot c\ra\, dt\\
    &\qquad \quad + \int_0^\ell - f\sum_{j=2}^n \la \del_{\frac{\bd c_j}{\bd s}} \frac{\bd c_j}{\bd s}, f^{-1} (\grad f)^\perp \ra -2f^{-2}\sum_{j=2}^n \la \grad f, e_j \ra \la e_j,\grad f\ra - f^{-2}\sum_{j=2}^n \la e_j,\grad f\ra^2\, dt  \\
    & \qquad \quad + \int_0^\ell - \ric(\dot c,\dot c) + f^{-2} \sum_{j=2}^n  \la \grad f,e_j\ra^2 +\frac{n-1}{4}\la \dot c, \nabla \ln f\ra^2 \, dt\\
    &= \int_0^\ell f^{-1} \lap f - f^{-1} \operatorname{Hess}_f(\dot c,\dot c)  + \sum_{j=2}^n \frac{\bd}{\bd t}\left[f \la \del_{\frac{\bd c_j}{\bd s}} \frac{\bd c_j}{\bd s}, \dot c\ra\right]  -2 f^{-2}\vert (\grad f)^\perp\vert^2   - \ric(\dot c,\dot c) +\frac{n-1}{4}\la \dot c, \nabla \ln f\ra^2\, dt.
    \end{align*}
    As before, we have 
    \[
    \operatorname{Hess}_f(\dot c,\dot c) = \la \del_{\dot c} \grad f, \dot c\ra = \frac{d^2 f}{dt^2} - \la \grad f,\del_{\dot c}\dot c\ra = \frac{d^2f}{dt^2} - f^{-1} \vert (\grad f)^\perp\vert^2.
    \]
    Therefore, combined with $\frac{n-1}{4}\la \dot c, \nabla \ln f\ra^2 = \frac{1}{2}\la \dot c, \nabla \ln f\ra^2\leq \frac{1}{2}|\nabla \ln f|^2$, we obtain 
    \begin{align*}
        0 &\le \int_0^\ell f^{-1} \lap f - f^{-1} \frac{d^2f}{dt^2} + \sum_{j=2}^n \frac{\bd}{\bd t}\left[f \la \del_{\frac{\bd c_j}{\bd s}} \frac{\bd c_j}{\bd s}, \dot c\ra\right]  - f^{-2}\vert (\grad f)^\perp\vert^2   - \ric(\dot c,\dot c) +\frac{1}{2}|\nabla \ln f|^2\, dt.
    \end{align*}
    Next, note that 
    \[
    f^{-1} \frac{d^2f}{dt^2} = \frac{d^2}{dt^2}(\ln f) + \left[\frac{d(\ln f)}{dt}\right]^2 = \frac{d^2}{dt^2}(\ln f) + f^{-2} \la \grad f,\dot c\ra^2.  
    \]
    It follows that 
    \begin{align*}
        0 &\le \int_0^\ell f^{-1}\lap f - f^{-2}\la \grad f,\dot c\ra^2 - f^{-2} \vert (\grad f)^\perp\vert^2 +\frac{1}{2}|\nabla \ln f|^2 - \ric(\dot c,\dot c) - \frac{d^2}{dt^2}(\ln f) + \sum_{j=2}^n \frac{\bd}{\bd t}\left[f \la \del_{\frac{\bd c_j}{\bd s}} \frac{\bd c_j}{\bd s}, \dot c\ra\right]\, dt\\
        &= \int_0^\ell f^{-1}\lap f - \frac{1}{2} \vert \grad \ln f\vert^2 - \ric(\dot c,\dot c) - \frac{d^2}{dt^2}(\ln f) + \sum_{j=2}^n \frac{\bd}{\bd t}\left[f \la \del_{\frac{\bd c_j}{\bd s}} \frac{\bd c_j}{\bd s}, \dot c\ra\right]\, dt\\
        &< \int_0^\ell - \frac{d^2}{dt^2}(\ln f) + \sum_{j=2}^n \frac{\bd}{\bd t}\left[f \la \del_{\frac{\bd c_j}{\bd s}} \frac{\bd c_j}{\bd s}, \dot c\ra\right]\, dt.
    \end{align*}
    Here we used the first assumption of the theorem to get the final inequality. Now we can apply the fundamental theorem of calculus to get 
    \begin{align*}
    0 &< -\bd_\eta (\ln f)(c(\ell)) - \bd_\eta (\ln f)(c(0)) - H_{\Sigma_2}(c(\ell)) - H_{\Sigma_1}(c(0))\\
    &= -\frac{\bd_\eta f}{f}(c(\ell)) - \frac{\bd_\eta f}{f}(c(0)) - H_{\Sigma_2}(c(\ell)) - H_{\Sigma_1}(c(0)). 
    \end{align*}
    The second assumption of the theorem implies that the previous line is equal to 0, and again we've reached a contradiction. 
\end{proof}

\section{Proof of the Main Theorem}\label{sec: proof of main theorem}

In this section, we prove our main result Theorem \ref{Theorem:Main-Theorem}.  We will proceed  case by case based on the value of $m$. 

\subsection{The Ricci Curvature Case} 

Note that the $1$-intermediate curvature is just the Ricci curvature. Therefore, the  $m = 1$ case of Conjecture \ref{generalized-Kpi1} is a well-known corollary of the Bonnet-Myers theorem. In this case, we do not need any restriction on the ambient dimension $n$. 

\begin{theorem}
Let $M^n$ be a closed manifold. Assume that the universal cover $\overline M$ of $M$ satisfies $H_n(\overline M,\Z) = 0$. Then $M$ does not admit a metric of positive Ricci curvature. 
\end{theorem}

\begin{proof}
We prove the contrapositive. Assume that $M$ admits a metric of positive Ricci curvature. Let $\overline M$ be the universal cover of $M$. By the Bonnet-Myers theorem, $\overline M$ is a closed, orientable $n$-dimensional manifold. Therefore $H_n(\overline M,\Z) \neq 0$.
\end{proof}

\subsection{The Scalar Curvature Case} We now show that the $m = n-1$ case of Conjecture \ref{generalized-Kpi1} is equivalent to the $K(\pi,1)$ conjecture. Since the $(n-1)$-intermediate curvature is equivalent to scalar curvature, it suffices to prove the following simple fact. 

\begin{theorem}
    Let $M^n$ be a closed manifold. Then the universal cover $\overline M$ of $M$ satisfies \begin{equation} \label{Equation:Homology}
    H_n(\overline M,\Z) = H_{n-1}(\overline M,\Z)  = \hdots = H_{2}(\overline M,\Z) = 0 
    \end{equation}
    if and only if $M$ is aspherical. 
\end{theorem}

\begin{proof}
    First suppose that $M$ is aspherical. Then $\overline M$ is contractible and so it is immediate that (\ref{Equation:Homology}) holds. Conversely, suppose that $\overline M$ satisfies (\ref{Equation:Homology}). Since $\overline M$ is an $n$-dimensional manifold, (\ref{Equation:Homology}) implies that $H_k(\overline M,\Z) = 0$ for all $k\ge 2$. 
    As $\overline M$ is simply connected, it then follows from the Hurewicz theorem \cite[Theorem 4.32]{hatcher2002algebraic} that $\pi_k(\overline M) = 0$ for all $k\ge 1$. The Whitehead theorem \cite[Theorem 4.5]{hatcher2002algebraic} now implies that $\overline M$ is contractible and so $M$ is aspherical. 
\end{proof}

Since the $K(\pi,1)$ conjecture is known to be true for $n\in \{3,4,5\}$, we have the following immediate corollary. 

\begin{corollary}
    Conjecture \ref{generalized-Kpi1} is true for $n\in \{3,4,5\}$ and $m = n-1$. 
\end{corollary}

\subsection{The Codimension Two Case}

In this section, we prove the $m=2$ case of Conjecture \ref{generalized-Kpi1} for $n\in \{3,4,5,6\}$.  Our argument combines a construction of Chodosh-Li \cite{chodosh2024generalized} with a diameter bound of Shen-Ye \cite{shen1996stable,shen1997geometry}.

\begin{theorem}
Let $M^n$ be a closed manifold of dimension $n\in \{3,4,5,6\}$, Assume that the universal cover $\overline M$ of $M$ satisfies $H_n(\overline M,\Z) = H_{n-1}(\overline M,\Z) = 0$. Then $M$ does not admit a metric of $2$-positive intermediate curvature. 
\end{theorem}

\begin{proof}
Let $M$ be as in the statement of the theorem. 
Assume for the sake of contradiction that $M$ admits a metric with positive 2-intermediate curvature. By scaling, we can suppose that the 2-intermediate curvature of $M$ is at least $2$. For a fixed large constant $L > 0$, we can apply Proposition \ref{linking} to find a geodesic line $\sigma$ in $\overline M$ and a closed manifold $\Lambda^{n-2}$ embedded in $\overline M$ such that $\Lambda$ is linked with $\sigma$ and $d(\sigma,\Lambda) \ge L$. Now let $\Sigma_1$ be the area minimizing minimal hypersurface in $\overline M$ with $\bd \Sigma_1 = \Lambda$. Then $\Sigma_1$ must intersect $\sigma$ and so there is a point $p\in \Sigma_1$ with $d(p,\bd \Sigma_1) \ge L$.  For $n\in \{3,4,5\}$, this contradicts the Shen-Ye \cite{shen1996stable} diameter estimate for stable minimal hypersurfaces in manifolds with positive 2-intermediate curvature (Theorem \ref{Theorem:Shen-Ye-BiRicci}) provided $L$ is chosen large enough. 

When $n=6$, we need to argue more carefully to get the diameter bound. After establishing the diameter bound, one then gets a contradiction in the same way. Choose $\tau < 1$ close enough to $1$ so that $(1-\tau) \vert \ric_M(v,v)\vert < 1$ for all unit vectors $v$. Let $\Sigma_1$ be constructed as above, and let $f > 0$ be the first eigenfunction of the stability operator so that 
\[
\lap_{\Sigma_1} f + \vert A_{\Sigma_1}\vert^2 f + \ric_{\overline M}(\nu,\nu)f \le 0.
\]
Since $\Sigma_1$ is five dimensional, Theorem \ref{Shen-Ye-Higher-Dimension} says that if 
\[
\ric_{\Sigma_1}(v,v) - \tau f^{-1} \lap_{\Sigma_1} f + \left[\tau - (1+\eps)\tau^2\right]  \vert \grad_{\Sigma_1} \ln f\vert^2 \ge \kappa > 0 
\]
for all unit vectors $v\in T_p\Sigma_1$, then the diameter of $\Sigma_1$ is bounded above in terms of $n$, $\kappa$, and $\eps$. Importantly, we can choose $\eps = \eps(\tau)$ close enough to 0 that $\tau - (1+\eps)\tau^2 \ge 0$.

Thus it is enough to show that 
\[
\ric_{\Sigma_1}(v,v) - \tau f^{-1}\lap_{\Sigma_1} f \ge \kappa > 0
\] 
for all unit vectors $v \in T_p\Sigma_1$. Let $e_1,e_2,e_3,e_4,e_5,e_6$ be an orthonormal basis for $T_p\overline{M}$ for which $e_2,e_3,e_4,e_5,e_6$ is an orthonormal basis for $T_p\Sigma_1$. We compute 
\begin{align*}
    \ric_{\Sigma_1}&(e_2,e_2) - \tau f^{-1}  \lap_{\Sigma_1} f \phantom{\int} \\
    &\ge \ric_{\Sigma_1}(e_2,e_2) + \tau \vert A_{\Sigma_1}\vert^2 + \tau \ric_{\overline M}(e_1,e_1)\\
    &= \ric_{\overline M}(e_2,e_2) + \tau \ric_{\overline M}(e_1,e_1) - R_{\overline M}(e_1,e_2,e_1,e_2)  + \tau \vert A_{\Sigma_1}\vert^2 + \sum_{j=3}^6 \left[  A^{\Sigma_1}_{22}A^{\Sigma_1}_{jj} - (A^{\Sigma_1}_{2j})^2\right]\\
    & \ge 2 + (\tau - 1) \ric_{\overline M}(e_1,e_1) +  \tau \vert A_{\Sigma_1}\vert^2 + \sum_{j=3}^6 \left[  A^{\Sigma_1}_{22}A^{\Sigma_1}_{jj} - (A^{\Sigma_1}_{2j})^2\right]\\ 
    &\ge 1 +  \tau \vert A_{\Sigma_1}\vert^2 + \sum_{j=3}^6 \left[  A^{\Sigma_1}_{22}A^{\Sigma_1}_{jj} - (A^{\Sigma_1}_{2j})^2\right].
\end{align*}
It remains to show that the second fundamental form terms give a non-negative contribution. We will drop the superscript $\Sigma_1$ in what follows. Since $\Sigma_1$ is minimal, we have 
\begin{align*}
\sum_{j=3}^6 \left[A_{22}A_{jj} - A_{2j}^2\right] &= -\sum_{i=3}^6 \sum_{j=3}^6 A_{ii}A_{jj} - \sum_{j=3}^6 A_{2j}^2  = - 2 \sum_{3\le i < j \le 6} A_{ii} A_{jj} - \sum_{j=3}^6 A_{jj}^2 - \sum_{j=3}^6 A_{2 j}^2. 
\end{align*}
Also we have 
\begin{align*}
     \vert A\vert^2 &\ge A_{22}^2 + \sum_{j=3}^6 A_{jj}^2 + 2 \sum_{j=3}^6 A_{2j}^2 \\
     &= \left(\sum_{j=3}^6 A_{jj}\right)^2 + \sum_{j=3}^6 A_{jj}^2 + 2 \sum_{j=3}^6 A_{2j}^2\\
     &= 2 \sum_{3\le i < j \le 6} A_{ii}A_{jj} + 2 \sum_{j=3}^6 A_{jj}^2 + 2\sum_{j=3}^6 A_{2j}^2.  
\end{align*}
So assuming $\tau \ge 1/2$ we get 
\begin{align}
\label{eqn1}
    \tau \vert A\vert^2 + \sum_{j=3}^6 \left[A_{22}A_{jj} - A_{2j}^2\right] & \ge (2\tau - 2) \sum_{3\le i < j \le 6} A_{ii}A_{jj} + (2\tau - 1) \sum_{j=3}^6 A_{jj}^2. 
\end{align}
Finally, note that 
\begin{align*}
    \sum_{j=3}^6 A_{jj}^2 = \frac 1 3 \sum_{3\le i < j \le 6} (A_{ii}^2 + A_{jj}^2) \ge \frac 2 3 \sum_{3\le i < j \le 6} \vert  A_{ii} A_{jj}\vert. 
\end{align*}
Combining this with (\ref{eqn1}), we obtain 
\[
\tau \vert A\vert^2 + \sum_{j=3}^6 \left[A_{22}A_{jj} - A_{2j}^2\right] \ge 0
\]
as long as 
\[
\frac{2-2\tau}{2\tau - 1} < \frac 2 3,
\]
which we can always ensure by choosing $\tau$ close enough to 1. This completes the proof.

\end{proof}

\subsection{The Case $n=5$ and $m=3$}


Next we show that Conjecture \ref{generalized-Kpi1} is true for $n = 5$ and $m = 3$. This will complete the proof of Theorem \ref{Theorem:Main-Theorem} for $n\in \{3,4,5\}$.

\begin{theorem} \label{thm: m3n5}
    Let $M^5$ be a closed manifold. Assume that the universal cover $\overline M$ of $M$ satisfies 
    \[
    H_{5}(\overline M,\Z) = H_{4}(\overline M,\Z) = H_{3}(\overline M,\Z) = 0.
    \]
    Then $M$ does not admit a metric with positive 3-intermediate curvature. 
\end{theorem}

    Let $M$ be as in the statement of the theorem. Assume for the sake of contradiction that $M$ admits a metric with positive $3$-intermediate curvature. By scaling, we can suppose the 3-intermediate curvature of $M$ is at least $1$. 
For a fixed large $L > 0$, we apply Proposition \ref{linking} to find a geodesic line $\sigma$ in $\overline M$ and a closed manifold $\Lambda^{n-2}$ embedded in $\overline M$ such that $\Lambda$ is linked with $\sigma$ and $d(\sigma,\Lambda) \ge 2 L$. Let $\Sigma_1$ be the area minimizing minimal hypersurface in $\overline M$ with $\bd \Sigma_1 = \Lambda$. Let $\eta$ be the unit normal to $\Sigma_1$. Since $\Sigma_1$ is area minimizing, there exists a function $u > 0$ on $\Sigma_1$ which satisfies $\lap_{\Sigma_1} u + (\vert A_{\Sigma_1}\vert^2 + \ric_{\overline M}(\eta,\eta))u \le 0$. 

    Next, we argue as in Chodosh-Li \cite{chodosh2024generalized}  to construct a suitable function $h$ on $\Sigma_1$.Define 
    \[
    \rho_0(x) = d_{\overline M}(x,\sigma), \quad x\in \overline M.
    \]
    Then let $\rho_1$ be a smooth approximation to $\rho_0$ which satisfies $\|\grad \rho_1(x)\| \le 2$ for all $x\in \overline M$, $\rho_1<1$ on $\sigma$, and $\rho_1>L+4\pi+1$ on $\bd\Sigma_1=\Lambda$. Let $\rho_2$ be the restriction of $\rho_1$ to $\Sigma_1$. Now choose $\eps > 0$ sufficiently small so that $L-\eps$, $L+2\pi(1+\eps) + \eps^2$, 
    and $L+4\pi + \eps$ are all regular values of $\rho_2$.  We can further ensure that $\Sigma_1 \cap B_{\overline M}(\sigma,L/4) \subset \{\rho_2 \le L-\eps\}$. More precisely, assume $L$ is chosen so that $L/4 < (L-\epsilon-1)/2$. Then for $q\in \Sigma_1 \cap B_{\overline M}(\sigma,L/4)$ we have 
    \[
    \rho_2(q)=\rho_1(q)\leq \sup_\sigma\rho_1 + (L/4)\cdot\|\nabla\rho_1\|_{L^\infty}\leq L-\epsilon,
    \]
    so $\Sigma_1 \cap B_{\overline M}(\sigma,L/4) \subset \{\rho_2 \le L-\eps\}$. Now define 
    \[
    \rho(x) = \frac{\rho_2(x) -L - 2\pi}{2\pi+\eps}, \quad x\in \Sigma_1.
    \]
    Then we have $\|\grad_{\Sigma_1} \rho(x)\| \le \frac 1 {\pi}$ for all $x\in \Sigma_1$. We define 
\begin{gather*}
        \tilde{\Sigma}_1 = \{-1 \le \rho \le 1\}, \\
        \bd_+ \tilde{\Sigma}_1 = \{\rho = -1\},\\
        \bd_- \tilde{\Sigma}_1 = \{\rho = 1\},\\
        \Omega_0 = \{-1 \le \rho <  \eps\}. 
    \end{gather*}
    By construction, $\bd_{\pm}\tilde{\Sigma}_1$ and $\bd \Omega_0$ are all smooth hypersurfaces in $\tilde{\Sigma}_1$, and $\Omega_0$ contains $\bd_+ \tilde{\Sigma}_1$, and $\bd \Sigma_1 \cap \tilde{\Sigma}_1 = \emptyset$. Finally, we define 
    \[
    h(x) = - \tan\left(\frac{\pi \rho(x)}{2}\right).
    \]
    Then $h$ is a smooth function on the interior of $\tilde{\Sigma}_1$ which satisfies $h(x) \to \pm\infty$ as $x\to \bd_{\pm}\tilde{\Sigma}_1$. Moreover, we have 
    \[
    \| \grad_{\Sigma_1} h(x)\| = \frac{\pi}{2}\sec^2\left(\frac{\pi \rho(x)}{2}\right) \| \grad_{\Sigma_1} \rho(x)\|  \le \frac {1} 2 \left(1 + \tan^2\left(\frac{\pi \rho(x)}{2}\right)\right) = \frac{1}{2}\left(1 + h^2\right) 
    \]
    for all $x\in \tilde{\Sigma}_1$. It follows that 
    \begin{equation}
    \label{eqn:h-inequality}
    1 + h^2 - 2\|\grad_{\Sigma_1} h\| \ge 0
    \end{equation}
    on $\tilde{\Sigma}_1$. Proposition \ref{prop: mu-bubble regularity} implies that there exists a minimizer $\Omega \subset \tilde{\Sigma}_1$ of the $\mu$-bubble functional 
    \[
\mathcal A(\Omega) = \int_{\bd \Omega} u\, d \mathcal H^3 - \int_{\tilde{\Sigma}_1} (\chi_\Omega - \chi_{\Omega_0}) hu\, d\mathcal H^4 
\]
such that $\Omega \operatorname{\Delta} \Omega_0$ is compactly contained in the interior of $\tilde{\Sigma}_1$. 

Let $\Sigma_2$ be a connected component of $\bd \Omega$, and  consider the slicing $\Sigma_2 \subset \Sigma_1 \subset \overline M$ together with the functions $u$ and $h$.

We are going to show that $\diam(\Sigma_2)$ is bounded from above by a universal constant that does not depend on $L$. 
Let $\nu$ be the unit normal vector to $\Sigma_2$ in $\Sigma_1$ pointing out of $\Omega$. According to Proposition \ref{proposition:variational-formulas}, the mean curvature of $\Sigma_2$ satisfies 
\[
H_{\Sigma_2} = h - u^{-1}\la \grad_{\Sigma_1} u,\nu\ra.
\]
In the next proposition, we re-arrange the second variation formula into a more convenient form.

\begin{prop}
\label{prop:simplified-second-variation}
    For any $\psi\in C^1(\Sigma_2)$ we have 
    \begin{align*}
    0 \le \int_{\Sigma_2} \vert \grad_{\Sigma_2} \psi\vert^2 u - \left(\vert A_{\Sigma_1}\vert^2 + \ric_{\overline{M}}(\eta,\eta) + \vert A_{\Sigma_2}\vert^2 + \ric_{\Sigma_1}(\nu,\nu)  - \frac 1 2 H_{\Sigma_2}^2 - \frac 1 2 + \frac{\lap_{\Sigma_2} u}{u}\right) \psi^2 u
    \end{align*}
\end{prop}

\begin{proof}
    The second variation formula (Proposition \ref{proposition:variational-formulas}) and $H_{\Sigma_2} = h - u^{-1}\la \grad_{\Sigma_1} u,\nu\ra$ give that 
    \[
    0 \le \int_{\Sigma_2} \vert \grad_{\Sigma_2} \psi \vert^2 u - (\vert A_{\Sigma_2}\vert^2 + \ric_{\Sigma_1}(\nu,\nu))\psi^2 u - \la \grad_{\Sigma_1} h,\nu\ra \psi^2 u - h\la \grad_{\Sigma_1} u,\nu\ra \psi^2 + (\lap_{\Sigma_1} u - \lap_{\Sigma_2} u)\psi^2.
    \]
    We now simplify this as in Chodosh-Li \cite{chodosh2024generalized}. First, observe that 
    \[
    \frac 1 2 H_{\Sigma_2}^2 \psi^2 u = \frac 1 2 u^{-1} \la \grad_{\Sigma_1} u,\nu\ra^2 \psi^2 - h\la \grad_{\Sigma_1} u,\nu\ra \psi^2 + \frac 1 2 h^2 \psi^2 u.
    \]
    Using this to eliminate the $h\la \grad_{\Sigma_1} u,\nu\ra \psi^2$ term in the previous inequality, we see that
    \begin{align*}
        0 &\le \int_{\Sigma_2} \vert \grad_{\Sigma_2} \psi\vert^2 u - (\vert A_{\Sigma_2}\vert^2 + \ric_{\Sigma_1}(\nu,\nu) - \frac 1 2 H_{\Sigma_2}^2)\psi^2 u + (\lap_{\Sigma_1} u - \lap_{\Sigma_2} u)\psi^2\\
    &\qquad\qquad  - \frac 1 2 \int_{\Sigma_2} \left(\la u^{-1} \grad_{\Sigma_1} u,\nu\ra^2 + h^2 + 2\la \grad_{\Sigma_1} h,\nu\ra\right)\psi^2 u.
    \end{align*}
    Since $\la u^{-1}\grad_{\Sigma_1} u,\nu\ra^2 \ge 0$ and $1 + h^2 + 2\la \grad_{\Sigma_1} h,\nu\ra \ge 0$ by (\ref{eqn:h-inequality}), it follows that 
    \begin{align*}
&0 \le \int_{\Sigma_2} \vert \grad_{\Sigma_2} \psi\vert^2 u - \left(\vert A_{\Sigma_2}\vert^2 + \ric_{\Sigma_1}(\nu,\nu) - \frac 1 2 H_{{\Sigma_2}}^2 - \frac 1 2\right)\psi^2 u + (\lap_{\Sigma_1} u - \lap_{\Sigma_2} u)\psi^2. 
\end{align*}
Finally, recalling that $\lap_{\Sigma_1} u + (\vert A_{\Sigma_1}\vert^2 + \ric_{\overline{M}}(\eta,\eta))u \le 0$, we get 
\[
0 \le \int_{\Sigma_2} \vert \grad_{\Sigma_2} \psi\vert^2 u - \left(\vert A_{\Sigma_1}\vert^2 + \ric_{\overline{M}}(\eta,\eta) + \vert A_{\Sigma_2}\vert^2 + \ric_{\Sigma_1}(\nu,\nu)  - \frac 1 2 H_{\Sigma_2}^2 - \frac 1 2 + \frac{\lap_{\Sigma_2} u}{u}\right) \psi^2 u,
\]
as needed.
\end{proof}

Now we can argue that the diameter of ${\Sigma_2}$ is bounded.

\begin{prop}
\label{Theorem:diameter-bound}
The diameter $\diam({\Sigma_2})$ of $\Sigma_2$ is bounded uniformly from above.     
\end{prop}

\begin{proof}
Note that Proposition \ref{prop:simplified-second-variation} implies that there is a function $w > 0$ on ${\Sigma_2}$ which satisfies 
\begin{equation}
\label{eqn:w-inequality}
\div_{\Sigma_2}(u \grad_{\Sigma_2} w) \le - \left(\vert A_{\Sigma_2}\vert^2 + \ric_{\Sigma_1}(\nu,\nu) + \vert A_{\Sigma_1}\vert^2 + \ric_{\overline{M}}(\eta,\eta) - \frac 1 2 H_{{\Sigma_2}}^2 - \frac 1 2 + \frac{\lap_{\Sigma_2} u}{u}\right) w u.
\end{equation} 
Our strategy is to apply Theorem \ref{Theorem:generalized-Bonnet-Myers} with $N = {\Sigma_2}$ and $f = uw$. 

    Fix a point $p\in {\Sigma_2}$ and let $\{e_1 = \eta,e_2 = \nu,e_3,e_4,e_5\}$ be an orthonormal frame at $p$ such that $\{e_2,e_3,e_4,e_5\}$ is an orthonormal basis for ${\Sigma_1}$ and $\{e_3,e_4,e_5\}$ is an orthonormal basis for ${\Sigma_2}$. We compute that 
    \begin{align*}
        &\ric_{\Sigma_2}(e_3,e_3) - \frac{\lap_{\Sigma_2}(uw)}{uw} + \frac{1}{2}\vert \grad_{\Sigma_2} \log(uw)\vert^2 \\
        &\qquad\qquad = \ric_{\Sigma_2}(e_3,e_3) - \frac{\div_{\Sigma_2}(u\grad_{\Sigma_2} w)}{uw} - \frac{\div_{\Sigma_2}(w\grad_{\Sigma_2} u)}{uw} + \frac{1}{2} \left\vert \frac{\grad_{\Sigma_2} u}{u} + \frac{\grad_{\Sigma_2} w}{w}\right\vert^2\\
        &\qquad \qquad \ge \ric_{\Sigma_2}(e_3,e_3) - \frac{\div_{\Sigma_2}(u\grad_{\Sigma_2} w)}{uw} - \frac{\lap_{\Sigma_2} u}{u}.  
    \end{align*}
    Therefore, by inequality (\ref{eqn:w-inequality}), we obtain 
    \begin{align*}
        &\ric_{\Sigma_2}(e_3,e_3) - \frac{\lap_{\Sigma_2}(uw)}{uw} + \frac{1}{2}\vert \grad_{\Sigma_2} \log(uw)\vert^2 \\
        &\qquad\qquad \ge  \ric_{\overline{M}}(e_1,e_1) + \ric_{\Sigma_1}(e_2,e_2) + \ric_{\Sigma_2}(e_3,e_3) + \vert A_{\Sigma_1}\vert^2 + \vert A_{\Sigma_2}\vert^2 - \frac{1}{2}H_{\Sigma_2}^2 - \frac 1 2. 
    \end{align*}
    It remains to get a lower bound on the right hand side. 

    According to \cite[Lemma 3.8]{brendle2024generalization}, we have
    \begin{align*}
    \ric_{\overline{M}}(e_1,e_1) + \ric_{\Sigma_1}(e_2,e_2) + \ric_{\Sigma_2}(e_3,e_3) = C_3(e_1,e_2,e_3) + \mathcal B
\end{align*}
where 
\[
\mathcal B =  \sum_{p=2}^3 \sum_{q=p+1}^5 \big(A_{\Sigma_1}(e_p,e_p)A_{\Sigma_1}(e_q,e_q) - A_{{\Sigma_1}}(e_p,e_q)^2\big) + \sum_{q=4}^5 \big(A_{\Sigma_2}(e_3,e_3)A_{\Sigma_2}(e_q,e_q) - A_{\Sigma_2}(e_3,e_q)^2\big). 
\]
Now, since ${\Sigma_1}$ is minimal, \cite[Lemma 3.11]{brendle2024generalization} implies that 
\[
\vert A_{\Sigma_1}\vert^2 + \sum_{p=2}^3 \sum_{q=p+1}^5 \big(A_{\Sigma_1}(e_p,e_p)A_{\Sigma_1}(e_q,e_q) - A_{{\Sigma_1}}(e_p,e_q)^2\big)  \ge 0. 
\]
Moreover, we have 
\begin{align*}
    &\vert A_{\Sigma_2}\vert^2 - \frac 1 2 H_{\Sigma_2}^2 + \sum_{q=4}^5 \big(A_{\Sigma_2}(e_3,e_3)A_{\Sigma_2}(e_q,e_q) - A_{\Sigma_2}(e_3,e_q)^2\big) \\
    &\qquad = \frac 1 2 (A^{\Sigma_2}_{33})^2 + \frac 1 2 (A^{\Sigma_2}_{44})^2 + \frac 1 2 (A^{\Sigma_2}_{55})^2 - A^{\Sigma_2}_{44}A^{\Sigma_2}_{55} + (A^{\Sigma_2}_{34})^2 + (A^{\Sigma_2}_{35})^2 + 2(A^{\Sigma_2}_{45})^2 \ge 0. 
\end{align*}
Hence we obtain 
\begin{align*}
&\ric_{\Sigma_2}(e_3,e_3) - \frac{\lap_{\Sigma_2}(uw)}{uw} + \frac{1}{2}\vert \grad_{\Sigma_2} \log(uw)\vert^2 \\
        &\qquad\qquad \ge C_3(e_1,e_2,e_3) + \mathcal B + \vert A_{\Sigma_1}\vert^2 + \vert A_{\Sigma_2}\vert^2 - \frac{1}{2}H_{\Sigma_2}^2 - \frac 1 2\\
        &\qquad \qquad \ge C_3(e_1,e_2,e_3) - \frac 1 2 \ge \frac 1 2. 
\end{align*}
Since $e_3$ was an arbitrary unit tangent vector to ${\Sigma_2}$ and $\Sigma_2$ is 3-dimensional, Shen-Ye's generalized Bonnet-Myers theorem (Theorem \ref{Theorem:generalized-Bonnet-Myers}) now gives the conclusion.  
\end{proof}

We can now finish the proof of Theorem \ref{thm: m3n5}. Since $\bd {\Sigma_2} = 0$ and $\diam({\Sigma_2})$ is uniformly bounded and $H_{3}(\overline M,\Z) = 0$, it follows from Proposition \ref{Proposition:uniformly-acyclic} that ${\Sigma_2}$ bounds within an $R$-neighborhood for some constant $R$ that depends only on $\overline M$. Applying this argument to each component of $\bd \Omega$, we see that $\bd \Omega$ bounds within its $R$-neighborhood in $\overline M$. Since $d(\bd \Omega, \sigma) \ge L/4$ and $\bd \Omega$ is linked with $\sigma$ by construction, this is a contradiction provided we select $L > 2R$. This completes the proof.

\subsection{The Case $n=6$ and $m=4$}
Next we prove the $n=6$ and $m=4$ case of Theorem \ref{Theorem:Main-Theorem}. Our argument combines the slice and dice procedure of Chodosh-Li \cite{chodosh2024generalized}, the diameter bound of Shen-Ye \cite{shen1996stable}, and also the Frankel type theorems we developed in Section \ref{sec: Frankel}.

\begin{theorem}\label{thm: m4n6}
    Let $M^6$ be a closed 6-dimensional manifold. Assume the universal cover $\overline{M}$ of $M$ satisfies $H_6(\overline{M},\mathbb Z)=H_5(\overline{M},\mathbb Z)=H_4(\overline{M},\mathbb Z)=H_3(\overline{M},\mathbb Z)=0$. Then $M$ does not admit a metric of positive $4$-intermediate curvature.
\end{theorem}

For $n=6$, $m=4$, let $\Sigma_2\subset\Sigma_1\subset \overline{M}^6$ be constructed as in the previous subsection, which are now of dimension $4$ and $5$. We consider the following weighted functional for $3$-dimensional $\Sigma_3\subset {\Sigma_2}$:
\[
\mathcal{A}_2(\Sigma_3):=\int_{\Sigma_3}uw.
\]
We will use the $\mathcal{A}_2$-functional to slice ${\Sigma_2}$ into a $4$-manifold $\hat{{\Sigma}}_2$ with simple third homology. Firstly, we prove that each boundary component of $\hat{{\Sigma}}_2$ has finite diameter. Without loss of generality, we assume the $4$-intermediate curvature of $M$ is at least $\frac{3}{2}$.

\begin{lem}\label{lem: diameterslice4}
    Suppose $\Sigma_3$ is a connected (two-sided) stable critical point of $\mathcal A_2$. Then $\diam \Sigma_3\leq 2\pi$ and $H_1(\Sigma_3)$ is finite.
\end{lem}

\begin{proof}
    By the first variation formula, we have
    \[
    H_{\Sigma_3}=-\langle \nabla_{{\Sigma_2}}\log(uw),\omega\rangle,
    \]
    where $H_{\Sigma_3}$ is the mean curvature of $\Sigma_3$ with respect to the unit normal $\omega$ of $\Sigma_3$ in ${\Sigma_2}$. 
    As before, let $\{e_i\}_{i=1}^6$ be an orthonormal basis of $T_p\overline{M}$ so that $e_j\perp\Sigma_j$ for $j=1,2,3$, i.e. $e_1=\eta, e_2=\nu, e_3=\omega$.
    
    By the second variation formula for $\mathcal A_2$, there exists a positive function $v\in C^{\infty}(\Sigma_3)$ such that
\[
-v^{-1}\Delta_{\Sigma_3}v\geq \|A_{\Sigma_3}\|^2+\ric_{{\Sigma_2}}(\omega,\omega)-\nabla_{{\Sigma_2}}^2\log(uw)(\omega,\omega)+\langle \nabla_{\Sigma_3}\log(uw),\nabla_{\Sigma_3}\log v\rangle.
\]
We let $f=uwv$, and our aim is to prove
\[
\ric_{\Sigma_3}(e,e)-f^{-1}\Delta_{\Sigma_3}f+\frac{1}{2}|\nabla_{\Sigma_3}\log f|^2\geq \frac{1}{2}>0,
\]
for any unit vector $e\in T\Sigma_3$.

Suppose $f_2$ is a smooth function defined on ${\Sigma_2}$. Then we have
\begin{align}\label{laponhyper}
\Delta_{{\Sigma_2}}f_2=\Delta_{\Sigma_3}f_2+\nabla_{{\Sigma_2}}^2f_2(\omega,\omega)-\langle\nabla_{{\Sigma_2}}f_2,\omega\rangle H_{\Sigma_3}.
\end{align}
By \eqref{eqn:w-inequality}, we have
\begin{align*}
    (uw)^{-1}&\Delta_{{\Sigma_2}}(uw)\\
    \leq& -\|A_{{\Sigma_1}}\|^2-\ric_{\overline{M}}(\eta,\eta)-\|A_{{\Sigma_2}}\|^2-\ric_{{\Sigma_1}}(\nu,\nu)+\frac{1}{2}H_{{\Sigma_2}}^2+\langle \nabla_{{\Sigma_2}}\log u,\nabla_{{\Sigma_2}}\log w\rangle+\frac{1}{2},\\
    =& -\|A_{{\Sigma_1}}\|^2-\ric_{\overline{M}}(\eta,\eta)-\|A_{{\Sigma_2}}\|^2-\ric_{{\Sigma_1}}(\nu,\nu)+\frac{1}{2}H_{{\Sigma_2}}^2+\frac{1}{2}\\
    &+\langle \nabla_{{\Sigma_2}}\log u,\omega\rangle\langle\nabla_{{\Sigma_2}}\log w,\omega\rangle+\langle \nabla_{\Sigma_3}\log u,\nabla_{\Sigma_3}\log w\rangle.
\end{align*}
Using \eqref{laponhyper} and the first variation formula, we can further write
\begin{align*}
    (uw)^{-1}\Delta_{\Sigma_3}(uw)=&(uw)^{-1}\Delta_{{\Sigma_2}}(uw)-\nabla_{{\Sigma_2}}^2\log (uw)(\omega,\omega)\\
    \leq &-\nabla_{{\Sigma_2}}^2\log(uw)(\omega,\omega)+\langle \nabla_{{\Sigma_2}}\log u,\omega\rangle\langle\nabla_{{\Sigma_2}}\log w,\omega\rangle+\langle\nabla_{\Sigma_3}\log u,\nabla_{\Sigma_3}\log w\rangle\\
    &-\|A_{{\Sigma_1}}\|^2-\ric_{\overline{M}}(\eta,\eta)-\|A_{{\Sigma_2}}\|^2-\ric_{{\Sigma_1}}(\nu,\nu)+\frac{1}{2}H_{{\Sigma_2}}^2+\frac{1}{2}.
\end{align*}
We now have that
\begin{align*}
    (uwv)^{-1}\Delta_{\Sigma_3}(uwv)=&(uw)^{-1}\Delta_{\Sigma_3}(uw)+v^{-1}\Delta_{\Sigma_3}v+2\langle \nabla_{\Sigma_3}\log(uw),\nabla_{\Sigma_3}\log v\rangle\\
    \leq & \langle\nabla_{{\Sigma_2}}\log u,\omega\rangle\langle\nabla_{{\Sigma_2}}\log w,\omega\rangle+\langle\nabla_{\Sigma_3}\log u,\nabla_{\Sigma_3}\log w\rangle+\langle\nabla_{{\Sigma_3  }}\log(uw),\nabla_{{\Sigma_1}}\log v\rangle\\
    &-\|A_{{\Sigma_1}}\|^2-\ric_{\overline{M}}(\eta,\eta)-\|A_{{\Sigma_2}}\|^2-\ric_{{\Sigma_1}}(\nu,\nu)-\|A_{\Sigma_3}\|^2-\ric_{{\Sigma_2}}(\omega,\omega)+\frac{1}{2}H_{{\Sigma_2}}^2+\frac{1}{2}\\
    \leq &\frac{1}{2}\left(\langle \nabla_{{\Sigma_2}}\log u,\omega\rangle+\langle \nabla_{{\Sigma_2}}\log w,\omega\rangle\right)^2+\langle\nabla_{\Sigma_3}\log u,\nabla_{\Sigma_3}\log w\rangle+\langle\nabla_{{\Sigma_3 }}\log(uw),\nabla_{{\Sigma_1}}\log v\rangle\\
    &-\|A_{{\Sigma_1}}\|^2-\ric_{\overline{M}}(\eta,\eta)-\|A_{{\Sigma_2}}\|^2-\ric_{{\Sigma_1}}(\nu,\nu)-\|A_{\Sigma_3}\|^2-\ric_{{\Sigma_2}}(\omega,\omega)+\frac{1}{2}H_{{\Sigma_2}}^2+\frac{1}{2}\\
    =&\langle\nabla_{\Sigma_3}\log u,\nabla_{\Sigma_3}\log w\rangle+\langle\nabla_{{\Sigma_3  }}\log(uw),\nabla_{{\Sigma_1}}\log v\rangle+\frac{1}{2}H_{{\Sigma_2}}^2+\frac{1}{2}H^2_{\Sigma_3}+\frac{1}{2}\\
    &-\|A_{{\Sigma_1}}\|^2-\ric_{\overline{M}}(\eta,\eta)-\|A_{{\Sigma_2}}\|^2-\ric_{{\Sigma_1}}(\nu,\nu)-\|A_{\Sigma_3}\|^2-\ric_{{\Sigma_2}}(\omega,\omega).
\end{align*}
We used the first variation in the last equality.

Summarizing the above computations, we have 
\begin{align*}
    \ric_{\Sigma_3}(e_4,e_4)&-(uwv)^{-1}\Delta_{\Sigma_3}(uwv)+\frac{1}{2}|\nabla_{\Sigma_3}\log(uwv)|^2\\
     \geq&\|A_{{\Sigma_1}}\|^2+\|A_{{\Sigma_2}}\|^2+\|A_{\Sigma_3}\|^2-\frac{1}{2}H_{{\Sigma_2}}^2-\frac{1}{2}H_{\Sigma_3}^2-\frac{1}{2}\\
     &+\ric_{\Sigma_3}(e_4,e_4)+\ric_{{\Sigma_2}}(\omega,\omega)+\ric_{{\Sigma_1}}(\nu,\nu)+\ric_{\overline{M}}(\eta,\eta)\\
     =&\|A_{{\Sigma_1}}\|^2+\|A_{{\Sigma_2}}\|^2+\|A_{\Sigma_3}\|^2-\frac{1}{2}H_{{\Sigma_2}}^2-\frac{1}{2}H_{\Sigma_3}^2-\frac{1}{2}\\
     &+C_4(e_1,e_2,e_3,e_4)+\sum_{p=2}^4\sum_{q=p+1}^6\left(A_{{\Sigma_1}}(e_p,e_p)A_{{\Sigma_1}}(e_q,e_q)-A_{{\Sigma_1}}(e_p,e_q)^2\right)\\
     &+\sum_{p=3}^4\sum_{q=p+1}^6\left(A_{{\Sigma_2}}(e_p,e_p)A_{{\Sigma_2}}(e_q,e_q)-A_{{\Sigma_2}}(e_p,e_q)^2\right)\\
     &+\sum_{q=5}^6\left(A_{\Sigma_3}(e_4,e_4)A_{\Sigma_3}(e_q,e_q)-A_{\Sigma_3}(e_4,e_q)^2\right)\\
     \geq& C_4(e_1,e_2,e_3,e_4)-\frac{1}{2}\geq 1.
\end{align*}
To get the last line above, we used the following estimates.
First, by \cite[Lemma 3.11]{brendle2024generalization}, we have
\[
\|A_{\Sigma_1}\|^2+\sum_{p=2}^4\sum_{p=q+1}^{6}\left(A_{\Sigma_1}(e_p,e_p)A_{\Sigma_1}(e_q,e_q)-A_{\Sigma_1}(e_p,e_q)^2\right)\geq 0.
\]
Moreover, via direct computations we have
\begin{align*}
    \|A_{\Sigma_2}\|^2 &-\frac{1}{2}H_{\Sigma_2}^2+\sum_{p=3}^4\sum_{q=p+1}^6\left(A_{\Sigma_2}(e_p,e_p)A_{\Sigma_2}(e_q,e_q)-A_{\Sigma_2}(e_p,e_q)^2\right)\\
    &=\frac{1}{2}\sum_{p=3}^6(A^{\Sigma_2}_{pp})^2 +\sum_{p=3}^4\sum_{q=p+1}^6(A_{pq}^{\Sigma_2})^2+2(A_{56}^{\Sigma_2})^2-A^{\Sigma_2}_{55}A_{66}^{\Sigma_2}\geq 0,
\end{align*}
and similarly
\begin{align*}
    \|A_{\Sigma_3}\|^2&-\frac{1}{2}H_{\Sigma_3}^2+\sum_{q=5}^6\left(A_{\Sigma_3}(e_4,e_4)A_{\Sigma_3}(e_q,e_q)-A_{\Sigma_3}(e_4,e_q)^2\right)\\
    &=\frac{1}{2}\sum_{p=4}^6(A_{pp}^{\Sigma_3})^2+\sum_{q=5}^6(A_{4q}^{\Sigma_3})^2+2(A_{56}^{\Sigma_3})^2-A_{55}^{\Sigma_3}A_{66}^{\Sigma_3}\geq 0.
\end{align*}
By the above computation, we know $\Sigma_3$ admits a metric with positive conformal Ricci curvature, and hence it has finite diameter; see Theorem \ref{Theorem:generalized-Bonnet-Myers}. As a direct corollary, the fundamental group of $\Sigma_3$ is finite and hence $H_1(\Sigma_3)$ is finite.
\end{proof}

We need the following slicing Lemma by Bamler-Li-Mantoulidis \cite{bamler2023decomposing}, which is a generalization of Chodosh-Li \cite[Lemma 20]{chodosh2024generalized}.

\begin{lem}\label{lem:surjective homology}
    There are $\Sigma_{3,1},\Sigma_{3,2},\dots,\Sigma_{3,k}\subset{\Sigma_2}$ pairwise disjoint two-sided stable critical points of $\mathcal A_2$ so that the manifold with boundary $\hat{\Sigma}_2:={\Sigma_2}\setminus(\cup_{i=1}^k\Sigma_{3,i})$ is connected and has $H_3(\partial\hat{{\Sigma}}_2)\to H_3(\hat{{\Sigma}}_2)$ surjective.
\end{lem}

\begin{proof}
    Suppose we have constructed $\Sigma_{3,1},\Sigma_{3,2},\cdots,\Sigma_{3,j}$ which are pairwise disjoint two-sided stable critical points of $\mathcal A_2$ and $M_j:=\Sigma_2\setminus (\cup_{i=1}^j\Sigma_{3,i})$ is connected. If the inclusion map $i: H_3(\partial M_j)\to H_3(M_j)$ is not surjective, then there exists $\Sigma_{3,j+1}$,  a closed connected stable two-sided critical point of $\mathcal{A}_2$, so the induction proceeds. To see that this process eventually terminates, we refer to Bamler-Li-Mantoulidis \cite[Lemma 2.5]{bamler2023decomposing}.
\end{proof}

Note that, by Poincaré duality and Lemma \ref{lem: diameterslice4}, we have
\[
H_2(\Sigma_3;\mathbb Z)=H^1(\Sigma_3;\mathbb Z)=\text{Hom}(H_1(\Sigma_3;\mathbb Z);\mathbb Z)=0.
\]

Given $\Omega\subset\hat{{\Sigma}}_2$, we write $\partial\Omega$ for its topological boundary, and we assume $\partial\Omega$ consists of smooth properly embedded surfaces in $\hat{{\Sigma}}_2$. 

\begin{lem}\label{lem:onecomponent}
    A connected component of $\hat{{\Sigma}}_2\setminus \Omega$ contains exactly one component of $\partial\Omega$.
\end{lem}

\begin{proof}
    Assuming the contrary, as in Chodosh-Li \cite{chodosh2024generalized}, there exists $\sigma$ an embedded $S^1$ such that $[\sigma]$ is not torsion in $H_1(\hat{{\Sigma}}_2)$.
    The long exact sequence in homology for $(\hat{{\Sigma}}_2,\partial\hat{{\Sigma}}_2)$ yields:
    \[
    H_3(\partial\hat{{\Sigma}}_2)\to H_3(\hat{{\Sigma}}_2)\to H_3(\hat{{\Sigma}}_2,\partial \hat{{\Sigma}}_2)\to H_2(\partial\hat{{\Sigma}}_2)=0.
    \]
    The final term vanishes since $\partial\hat{{\Sigma}}_2$ consisits of components with vanishing second homology group. Combining with Lemma \ref{lem:surjective homology},  we conclude that 
    $H_3(\hat{{\Sigma}}_2,\partial \hat{{\Sigma}}_2)=0$. Poincaré duality implies that $H^1(\hat{{\Sigma}}_2)=0$, and so the universal coefficient theorem implies that $H_1(\hat{{\Sigma}}_2)$ is torsion. This is a contradiction.
\end{proof}

Next, we proceed to the dice-procedure as in Chodosh-Li \cite{chodosh2024generalized}.
We consider the following $\mu$-bubble functional:
\[
\mathcal A_3(\Omega)=\int_{\partial^*\Omega}uw-\int_{\Omega}(\chi_{\Omega}-\chi_{\Omega_0})uwh,
\]
where $h$ satisfies 
\[
1+h^2(x)-2|\nabla h|\geq 0,
\]
and is to be specified later.

\begin{prop} \label{prop: conformal positivity}
    Suppose $\Upsilon$ is a component of a stable, free boundary $\mu$-bubble of $\mathcal A_3$. Then there exists a conformal factor such that the corresponding conformal Ricci curvature on $\Upsilon$ is positive and the boundary is minimal after the conformal change.
\end{prop}

\begin{proof}
    By the first variation, the mean curvature of $\Upsilon$ with respect to the (outer) unit normal $\nu_{\Upsilon}$ of $\Upsilon$ in ${\Sigma_2}$ is given by
    \[
    H_{\Upsilon}=-\langle \nabla_{{\Sigma_2}}\log(uw),\nu_{\Upsilon}\rangle+h.
    \]
    For any smooth function $\psi$ on $\Upsilon$, the second variation formula gives
    \begin{align*}
        0\leq& \int_{\Upsilon}|\nabla_{\Upsilon}\psi|^2uw-\left(\|A_{\Upsilon}\|^2+\ric_{{\Sigma_2}}(\nu_{\Upsilon},\nu_{\Upsilon})\right)\psi^2uw-\langle \nabla_{{\Sigma_2}}h,\nu_{\Upsilon}\rangle\psi^2uw-h\langle \nabla_{{\Sigma_2}}(uw),\nu_{\Upsilon}\rangle\psi^2\\
        &+\int_{\Upsilon}\left(\Delta_{{\Sigma_2}}(uw)-\Delta_{\Upsilon}(uw)\right)\psi^2-\int_{\partial\Upsilon}A_{\partial\hat{\Sigma}_2}(\omega,\omega)\psi^2uw,
    \end{align*}
    where 
    $\omega$ is the unit normal of $\partial\hat{\Sigma}_2$. Since $\partial\Upsilon$ meets $\partial\hat{\Sigma}_2$ orthogonally, $\omega$ is also the unit outer normal of $\partial\Upsilon$ in $\Upsilon$.

    Since the components of $\partial\hat{\Sigma}_2$ are $\mu$-bubbles of $\mathcal A_2$, we see from the first variation formula that
    \[
    H_{\partial\hat{\Sigma}_2}=-\langle \nabla_{{\Sigma_2}}\log(uw),\omega\rangle,
    \]
    where $H_{\partial\hat{\Sigma}_2}$ is the mean curvature of $\partial\hat{\Sigma}_2$ with respect to $\omega$. 
    Noting $\partial\Upsilon$ meets $\partial\hat{\Sigma}_2$ orthogonally, we can denote by $H_{\bd\Upsilon}$ the mean curvature of $\bd\Upsilon$ with respect to $\omega$, and write
    \[
    H_{\partial\hat{\Sigma}_2}=H_{\partial\Upsilon}+A_{\partial\hat{\Sigma}_2}(\omega,\omega).
    \]
    Since $\omega$ lies in the tangent space of $\Upsilon$, we rewrite the first variation formula of $\partial\hat{\Sigma}_2$ at the intersection points with $\partial\Upsilon$ as
    \[
    H_{\partial\hat{\Sigma}_2}=-\langle\nabla_{\Upsilon}\log(uw),\omega\rangle.
    \]
Now we have that
\begin{align*}
    \int_{\Upsilon}|\nabla_{\Upsilon}\psi|^2uw&-\int_{\partial\Upsilon}A_{\partial\hat{\Sigma}_2}(\omega,\omega)\psi^2uw\\
    =&\int_{\Upsilon}-uw\psi\Delta_{\Upsilon}\psi-\psi\langle \nabla_{\Upsilon}\psi,\nabla_{\Upsilon}(uw)\rangle\\
    &+\int_{\partial\Upsilon}\left(\langle\nabla_{\Upsilon}\psi,\omega\rangle uw+\langle\nabla_{\Upsilon}(uw),\omega\rangle\psi-\langle\nabla_{\Upsilon}(uw),\omega\rangle\psi-A_{\partial\hat{\Sigma}_2}(\omega,\omega)\psi uw\right)\psi\\
    =&\int_{\Upsilon}-uw\psi\Delta_{\Upsilon}\psi-\psi\langle\nabla_{\Upsilon}\psi,\nabla_{\Upsilon}(uw)\rangle\\
    &+\int_{\partial\Upsilon}\left(\langle\nabla_{\Upsilon}(uw\psi),\omega\rangle+\left(H_{\partial\hat{\Sigma}_2}-A_{\partial\hat{\Sigma}_2}(\omega,\omega)\right)uw\psi\right)\psi\\
    =&\int_{\Upsilon}-uw\psi\Delta_{\Upsilon}\psi-\psi\langle\nabla_{\Upsilon}\psi,\nabla_{\Upsilon}(uw)\rangle+\int_{\partial\Upsilon}\left(\langle\nabla_{\Upsilon}(uw\psi),\omega\rangle+H_{\partial\Upsilon}uw\psi\right)\psi.
\end{align*}
Summarizing the above computations, we can rewrite the second variation formula of $\Upsilon$ as
\begin{align*}
    0\leq& \int_{\Upsilon}-uw\psi\Delta_{\Upsilon}\psi-\psi\langle\nabla_{\Upsilon}\psi,\nabla_{\Upsilon}(uw)\rangle-\left(\|A_{\Upsilon}\|^2+\ric_{{\Sigma_2}}(\nu_{\Upsilon},\nu_{\Upsilon})\right)\psi^2uw\\
    &\int_{\Upsilon}-\langle\nabla_{{\Sigma_2}}h,\nu_{\Upsilon}\rangle\psi^2uw-h\langle\nabla_{{\Sigma_2}}(uw),\nu_{\Upsilon}\rangle\psi^2+\left(\Delta_{{\Sigma_2}}(uw)-\Delta_{\Upsilon}(uw)\right)\psi^2\\
    &+\int_{\partial\Upsilon}\left(\langle \nabla_{\Upsilon}(uw\psi),\omega\rangle+H_{\partial\Upsilon}uw\psi\right)\psi.
\end{align*}
The first variation formula implies the following equality
\[
\frac{1}{2}H_{\Upsilon}^2\psi^2uw=\frac{1}{2}(uw)^{-1}\langle \nabla_{{\Sigma_2}}(uw),\nu_{\Upsilon}\rangle^2\psi^2-h\langle \nabla_{{\Sigma_2}}(uw),\nu_{\Upsilon}\rangle\psi^2+\frac{1}{2}h^2\psi^2uw.
\]
As in Chodosh-Li \cite{chodosh2024generalized}, we use the above equality and the fact that
\[
1-2|\nabla_{{\Sigma_2}}h|+h^2\geq 0
\]
to further simplify the second variation formula as follows:
\begin{align*}
    0\leq& \int_{\Upsilon}-uw\psi\Delta_{\Upsilon}\psi-\psi\langle \nabla_{\Upsilon}\psi,\nabla_{\Upsilon}(uw)\rangle-\left(\|A_{\Upsilon}\|^2+\ric_{{\Sigma_2}}(\nu_{\Upsilon},\nu_{\Upsilon})-\frac{1}{2}H_{\Upsilon}^2-\frac{1}{2}\right)\psi^2uw\\
    &+\int_{\Upsilon}\left(\Delta_{{\Sigma_2}}(uw)-\Delta_{\Upsilon}(uw)\right)\psi^2-\frac{1}{2}\langle \nabla_{{\Sigma_2}}\log(uw),\nu_{\Upsilon}\rangle^2\psi^2uw+\int_{\partial\Upsilon}\left(\langle\nabla_{\Upsilon}(uw\psi),\omega\rangle+H_{\partial\Upsilon}uw\psi\right)\psi.
\end{align*}
Thus there exists a smooth positive function $v_2$ defined on $\Upsilon$ satisfying
\begin{equation} \label{eq: v2}
\begin{aligned}
-v_2^{-1}\Delta_{\Upsilon}v_2
\geq&\langle \nabla_{\Upsilon}\log v_2,\nabla_{\Upsilon}\log(uw)\rangle -(uw)^{-1}\left(\Delta_{{\Sigma_2}}(uw)-\Delta_{\Upsilon}(uw)\right)\\
&+\frac{1}{2}\langle \nabla_{{\Sigma_2}}\log(uw),\nu_{\Upsilon}\rangle^2
 +\|A_{\Upsilon}\|^2+\ric_{{\Sigma_2}}(\nu_{\Upsilon},\nu_{\Upsilon})-\frac{1}{2}H_{\Upsilon}^2-\frac{1}{2}.
\end{aligned}
\end{equation}
and it satisfies the following boundary condition on $\partial\Upsilon$
\begin{align}\label{Eq: boundary condition of FBB}
    \langle \nabla_{\Upsilon}(uwv_2),\omega\rangle+H_{\partial\Upsilon}uwv_2=0.
\end{align}
The boundary condition implies that $\partial\Upsilon$ is minimal under the conformal change $(uwv_2)^2g$.

By Proposition \ref{prop:simplified-second-variation}, we have
\begin{equation}\label{eq: uw}
\begin{aligned}
    (uw)^{-1}&\Delta_{{\Sigma_2}}(uw)\\
    \leq& \langle \nabla_{{\Sigma_2}}\log u,\nabla_{{\Sigma_2}}\log w\rangle -\|A_{\Sigma_1}\|^2-\ric_{\overline{M}}(\eta,\eta)-\|A_{{\Sigma_2}}\|^2-\ric_{\Sigma_1}(\nu,\nu)+\frac{1}{2}H_{{\Sigma_2}}^2+\frac{1}{2}.
\end{aligned}
\end{equation}
Combining \eqref{eq: v2} and \eqref{eq: uw}, we obtain
\begin{align*}
    (uwv_2)^{-1}&\Delta_{\Upsilon}(uwv_2)\\
    =&(uw)^{-1}\Delta_{\Upsilon}(uw)+v_2^{-1}\Delta_{\Upsilon}(v_2)+2\langle \nabla_{\Upsilon}\log(uw),\nabla_{\Upsilon}\log v_2\rangle\\
    \leq& (uw)^{-1}\Delta_{{\Sigma_2}}(uw)+\langle \nabla_{\Upsilon}\log(uw),\nabla_{\Upsilon}\log v_2\rangle -\|A_{\Upsilon}\|^2-\ric_{{\Sigma_2}}(\nu_3,\nu_3)+\frac{1}{2}H_{\Upsilon}^2+\frac{1}{2} -\frac{1}{2}\langle \nabla_{{\Sigma_2}}\log(uw),\nu_{\Upsilon}\rangle^2\\
    \leq& \langle \nabla_{{\Sigma_2}}\log u,\nu_{\Upsilon}\rangle \langle \nabla_{{\Sigma_2}}\log w,\nu_{\Upsilon}\rangle+\langle \nabla_{\Upsilon}\log u,\nabla_{\Upsilon}\log w\rangle -\frac{1}{2}\langle \nabla_{{\Sigma_2}}\log(uw),\nu_{\Upsilon}\rangle^2 +\langle \nabla_{\Upsilon}\log (uw),\nabla_{\Upsilon}\log v_2\rangle \\
    &-\|A_{\Sigma_1}\|^2-\|A_{{\Sigma_2}}\|^2-\|A_{\Upsilon}\|^2-\ric_{\overline{M}}(\eta,\eta)-\ric_{\Sigma_1}(\nu,\nu)-\ric_{{\Sigma_2}}(\nu_{\Upsilon},\nu_{\Upsilon})\\
    &+\frac{1}{2}H_{{\Sigma_2}}^2+\frac{1}{2}H_{\Upsilon}^2+1,\\
    \leq &\langle \nabla_{\Upsilon}\log(uw),\nabla_{\Upsilon}\log v_2\rangle +\langle \nabla_{\Upsilon}\log u,\nabla_{\Upsilon}\log w\rangle -\|A_{\Sigma_1}\|^2-\|A_{{\Sigma_2}}\|^2-\|A_{\Upsilon}\|^2\\
    & -\ric_{\overline{M}}(\eta,\eta)-\ric_{\Sigma}(\nu,\nu)-\ric_{{\Sigma_2}}(\nu_{\Upsilon},\nu_{\Upsilon})+\frac{1}{2}H_{{\Sigma_2}}^2+\frac{1}{2}H_{\Upsilon}^2+1.
\end{align*}

Supposing $e_4$ is a unit vector in $T\Upsilon$, we have
\begin{align*}
    \ric_{\Upsilon}(e_4,e_4)&-(uwv_2)^{-1}\Delta_{\Upsilon}(uwv_2)+\frac{1}{2}|\nabla_{\Upsilon}\log(uwv_2)|^2\\
    \geq& \frac{1}{2}|\nabla_{\Upsilon}\log(uwv_2)|^2-\langle \nabla_{\Upsilon}\log(uw),\nabla_{\Upsilon}\log v_2\rangle -\langle \nabla_{\Upsilon}\log u,\nabla_{\Upsilon}\log w\rangle \\
    &+\|A_{\Sigma_1}\|^2+\|A_{{\Sigma_2}}\|^2+\|A_{\Upsilon}\|^2-\frac{1}{2}H_{{\Sigma_2}}^2-\frac{1}{2}H_{\Upsilon}^2-1\\
    &+\ric_{\overline{M}}(\eta,\eta)+\ric_{\Sigma_1}(\nu,\nu)+\ric_{{\Sigma_2}}(\nu_{\Upsilon},\nu_{\Upsilon})+\ric_{\Upsilon}(e_4,e_4)\\
    \geq & \|A_{{\Sigma_1}}\|^2+\|A_{\Sigma_2}\|^2+\|A_{\Upsilon}\|^2-\frac{1}{2}H_{{\Sigma_2}}^2-\frac{1}{2}H_{\Upsilon}^2-1\\
     &+\ric_{\Upsilon}(e_4,e_4)+\ric_{{\Sigma_2}}(\omega,\omega)+\ric_{\Sigma_1}(\nu,\nu)+\ric_{\overline{M}}(\eta,\eta)\\
     =&\|A_{{\Sigma_1}}\|^2+\|A_{\Sigma_2}\|^2+\|A_{\Upsilon}\|^2-\frac{1}{2}H_{{\Sigma_2}}^2-\frac{1}{2}H_{\Upsilon}^2-1\\
     &+C_4(e_1,e_2,e_3,e_4)+\sum_{p=2}^4\sum_{q=p+1}^6\left(A_{\Sigma_1}(e_p,e_p)A_{\Sigma_1}(e_q,e_q)-A_{\Sigma_1}(e_p,e_q)^2\right)\\
     &+\sum_{p=3}^4\sum_{q=p+1}^6\left(A_{{\Sigma_2}}(e_p,e_p)A_{{\Sigma_2}}(e_q,e_q)-A_{{\Sigma_2}}(e_p,e_q)^2\right)\\
     &+\sum_{q=5}^6\left(A_{\Upsilon}(e_4,e_4)A_{\Upsilon}(e_q,e_q)-A_{\Upsilon}(e_4,e_q)^2\right)\\
     \geq& C_4(e_1,e_2,e_3,e_4)-1\geq \frac{1}{2}.
\end{align*}
We have completed the proof.
\end{proof}

\begin{corollary} \label{cor: freeboundarybubbble}
    Suppose $\Upsilon$ is a component of a  free boundary $\mu$-bubble of $\mathcal A_3$. Then it has connected boundary and bounded diameter.
\end{corollary}

\begin{proof}
    Given the previous proposition, the diameter estimate follows from Theorem \ref{Theorem:generalized-Bonnet-Myers}, and the fact that the boundary is connected follows from Theorem \ref{Theorem:ConformalRicciFrankel} and \eqref{Eq: boundary condition of FBB}.
\end{proof}

We can summarize the dice procedure as follows.

\begin{prop} \label{prop:dice procedure}
    Suppose $p\in \hat{\Sigma}_2$ is a fixed interior point, and assume further that $B_{\eps}(p)\subset\subset \hat{\Sigma}_2$. There exists a finite number $k$ and open connected domains $\{\Omega_i\}_{i=1}^k$, 
    \[
    B_{\eps}(p)=\Omega_1\subset \Omega_2,\subset\cdots\subset\Omega_k=\hat{\Sigma}_2,
    \]
    with the following properties:
    \begin{enumerate}
        \item $d_{\hat{\Sigma}_2}(\partial \Omega_{i+1},\partial\Omega_i)\geq \frac{2}{3}\pi$;
        \item Each component of $\Omega_{i+1}\setminus\Omega_i$ has diameter at most $10\pi$;
        \item Any component $\Upsilon\subset \partial\Omega_j$ has diameter $\diam\Upsilon\leq \pi$;
        \item Each component of $\partial\Omega_j$ is either a closed manifold with finite fundamental group or a compact manifold with connected boundary in $\partial\hat{\Sigma}_2$.
    \end{enumerate}
\end{prop}

\begin{proof}
    The proof follows from induction. Suppose we have constructed $\Omega_1,\Omega_2,\cdots, \Omega_j$ satisfying the above assumptions and $\Omega_j\neq \hat{\Sigma}_2$. Smooth $d(\cdot,\Omega_j)$ to a function $\rho$ such that $d(\cdot,\Omega_j)\leq \rho\leq \frac{3}{2}d(\cdot,\Omega_j)$ and $\rho|_{\Omega_j}=0$. Take $h$ to be 
    \[
    h(x)=-\tan\left(\frac{1}{3}(\rho-\pi)-\frac{\pi}{2}\right),
    \]
    and note that 
    \[
    1+h^2-2|\nabla h|\geq 0.
    \]
    Minimize $\hat{\mathcal A}_3$ to get a $\mu$-bubble  $\Omega_{j+1}$.

    By definition of $h$, $\Omega_{j+1}$ satisfies condition $(1)$. By Lemma \ref{lem:onecomponent} and Lemma \ref{lem: diameterslice4}, it satisfies condition $(2)$. By Corollary \ref{cor: freeboundarybubbble} and Lemma \ref{lem: diameterslice4}, it also satisfies condition $(3)$ and $(4)$.
    Since condition $(1)$ is satisfied by the sets we constructed, the induction process must terminate after finitely many steps.
\end{proof}

We now have all the ingredients we need to prove the last case of the Main Theorem \ref{Theorem:Main-Theorem}.

\begin{proof}[Proof of Theorem \ref{thm: m4n6}]
    Let $M$ be as in the statement of the theorem. Assume for the sake of contradiction that $M$ admits a metric with positive $4$-intermediate curvature. We can assume the $4$-intermediate curvature of $M$ is at least $\frac{3}{2}$.

    Arguing as in the proof of Theorem \ref{thm: m3n5}, we obtain $5$-dimensional $\Sigma_1$ and $4$-dimensional $\Sigma_2$. By Lemma \ref{lem: diameterslice4}, Lemma \ref{lem:surjective homology} and Proposition \ref{prop:dice procedure}, there exists a set of disjoint embedded closed $3$ manifolds $\Sigma_{3,1},\Sigma_{3,2}\cdots, \Sigma_{3,k}\subset \Sigma_2$ with $\diam \Sigma_{3,i}\leq \pi$. Moreover, there exists a set of embedded compact $3$-manifolds $\Upsilon_1,\Upsilon_2,\cdots,\Upsilon_l\subset \Sigma_2$ such that, $\diam \Upsilon_j\leq \pi$, $\partial\Upsilon_j$ is connected and contained in $\cup_{i=1}^k\Sigma_{3,i}$, and the interiors of $\Upsilon_j$ are pairwise disjoint with each other and disjoint with each $\Sigma_{3,i}$. Also, each component $K_1,K_2,\cdots, K_s$ of 
    \[
    \Sigma_2\setminus \left(\left(\cup_{i=1}^k\Sigma_{3,i}\right)\cup\left(\cup_{j=1}^s\Upsilon_j\right)\right)
    \]
    has diameter bounded by $10\pi$.

    We write the boundary components of $K_j$ as $\hat{\Upsilon}_j^1,\hat{\Upsilon}^2_j,\cdots,\hat{\Upsilon}_j^{n(j)}$.  By Proposition \ref{Proposition:uniformly-acyclic}, there exists a positive $R>0$, independent of $L$, such that we can fill-in $\hat{\Upsilon}_j^i$ by $\hat{K}_j^i$ with extrinsic diameter at most $R$.
    Then 
    \[
    K_j-\hat{K}_j^1-\cdots -\hat{K}_j^{n(j)}
    \]
    is a cycle with extrinsic diameter at most $2R+10\pi$.  By Proposition \ref{Proposition:uniformly-acyclic}, there exists a $5$-chain $\tilde{\Gamma}_j$ with extrinsic diameter bounded by $\tilde{R}$, which is independent of $L$, such that
    \[
    \partial\tilde{\Gamma}_j=K_j-\hat{K}_j^1-\cdots -\hat{K}_j^{n(j)}.
    \]
    Note that
    \[
    \Sigma_2-\sum_{j=1}^s\partial \tilde{\Gamma}_j=\sum_{j=1}^s\sum_{i=1}^{n(j)}\hat{K}_j^i.
    \]
    Since $\Upsilon_i$ has connected boundary, there exists an index $u(i,j)\in\{1,2,\cdots,k\}$ so that $\hat{\Upsilon}_j^i$ only intersects with $\Sigma_{3,u(i,j)}$ but not any of the other components of $\partial\hat{\Sigma}_2$. We group the $\hat{K}_j^i$ by $u(i,j)=a$ for $a\in\{1,2,\cdots,k\}$. Then
    \[
    \sum_{\{i,j: u(i,j)=a\}}\hat{K}_j^i
    \]
    is a cycle of diameter at most $2R+\pi$. Furthermore,
    \[
    \partial\left[\sum_{a=1}^k\sum_{\{i,j:u(i,j)=a\}}\hat{K}_j^i\right]=0.
    \]
    Therefore, by Proposition \ref{Proposition:uniformly-acyclic}, there exists $\hat{R}>0$, independent of $L$, such that there exists a $5$-chain $\Theta_a$ with extrinsic diameter at most $\hat{R}$ satisfying
    \[
    \partial\Theta_a=\sum_{\{i,j:u(i,j)=a\}}\hat{K}_j^i.
    \]
    In conclusion, we have
    \[
    \Sigma_2=\partial\left[\sum_{j=1}^s\tilde{\Gamma}_j+\sum_{a=1}^k\Theta_a\right],
    \]
    where each term in the sum has uniform bounded diameter as $L\to\infty$. A contradiction is achieved.
\end{proof}

\subsection{The Case $n=6$ and $m=3$}  Finally, we prove the case where $n=6$ and $m=3$. This will complete the proof of Theorem \ref{Theorem:Main-Theorem}.  This case is somewhat subtle and requires a delicate analysis. 

\begin{theorem}
    \label{Theorem:n6m3}Let $M^6$ be a closed manifold. Assume that the universal cover $\overline M$ of $M$ satisfies $H_6(\overline M,\Z) = H_5(\overline M,\Z) =H_4(\overline M,\Z) = 0$. Then $M$ does not admit a metric with positive 3-intermediate curvature.
\end{theorem}

Let $M$ be as in the statement of the theorem. Assume for the sake of contradiction that $M$ admits a metric with positive $3$-intermediate curvature. By scaling, we can suppose the 3-intermediate curvature of $M$ is at least $3$. 
For a fixed large $L > 0$, we apply Proposition \ref{linking} to find a geodesic line $\sigma$ in $\overline M$ and a closed manifold $\Lambda^{n-2}$ embedded in $\overline M$ such that $\Lambda$ is linked with $\sigma$ and $d(\sigma,\Lambda) \ge 2 L$. Let $\Sigma_1$ be the area minimizing minimal hypersurface in $\overline M$ with $\bd \Sigma_1 = \Lambda$. Let $\eta$ be the unit normal to $\Sigma_1$. Since $\Sigma_1$ is area minimizing, there exists a function $u > 0$ on $\Sigma_1$ which satisfies $\lap_{\Sigma_1} u + (\vert A_{\Sigma_1}\vert^2 + \ric_{\overline M}(\eta,\eta))u \le 0$. 

We are now going to construct a $\mu$-bubble in $\Sigma_1$.  The construction depends on a choice of several parameters $a,\tau,\eps,\alpha$ which depend only on $M$ and not on $L$.  First, we select
\[
0.9748\approx\frac{\sqrt{390}}{10}-1 < a < 1
\]
close enough to 1 so that the following two conditions hold:
\begin{itemize}
\item[(i)] $(1-a^2) \vert R_M(X,Y,X,Y)\vert \le \frac 1 {100}$ for all unit tangent vectors $X,Y$ to $M$;
\item[(ii)] The matrix 
\[
\begin{pmatrix}
\frac{1}{a}-\frac{9}{16} & \frac 1 2 - \frac 1 a & \frac{\sqrt 3}{8}\\ \\
\frac 1 2 - \frac 1 a & \frac{1}{a} & 0 \\ \\
\frac{\sqrt 3}{8} & 0 & \frac 1 4
\end{pmatrix}
\]
is positive definite. 
\end{itemize}
To see that (ii) is possible, note that the determinant of this matrix is $-\frac 1 {16} + \frac{1}{16a}$ which is positive for $a < 1$. Also when $a=1$, by direct computation,  the matrix 
\[
\begin{pmatrix}
\frac{7}{16} & -\frac 1 2  & \frac{\sqrt 3}{8}\\ \\
-\frac 1 2  & 1 & 0 \\ \\
\frac{\sqrt 3}{8} & 0 & \frac 1 4
\end{pmatrix}
\]
has eigenvalues 
\[
\lambda_1 = \frac{1}{32}(27+\sqrt{217}) > 0,\quad   \lambda_2 = \frac 1{32}(27-\sqrt{217}) > 0, \quad \lambda_3 = 0.
\]
Therefore, by continuity, a suitable choice of $a$ is possible. Next, we select $a < \tau < 1$ and $\alpha > 0$ so that the following condition holds: 
\begin{itemize}
    \item[(iii)] The matrix
    \[
\begin{pmatrix}
\frac{3}{16}-\frac 3 4 \tau + \frac \tau a & \frac \tau 2 - \frac \tau a & \frac{\sqrt 3}{8}\\ \\
\frac \tau 2 - \frac \tau a & \frac{\tau}{a} - \alpha & 0 \\ \\
\frac{\sqrt 3}{8} & 0 &  \tau - \frac{3}{4}
\end{pmatrix}
\]
is positive definite. 
\end{itemize}Since this matrix reduces to the one in condition (ii) when $\tau = 1$ and $\alpha = 0$, such a choice is again possible by continuity. Finally, since $\tau < 1$, we can select $\eps > 0$ so that 
\begin{itemize}
    \item[(iv)] $\displaystyle 
\tau - \left(\frac 3 4 + \eps\right) \tau^2 \ge \frac \tau 4.$
\end{itemize}
The reason for selecting this choice of parameters will become apparent in the proof. 

Next, consider the differential equation 
\[
\begin{cases}
k'(t) = -1 - \frac{\alpha}{2}k(t)^2,\\
k(0) = 0. 
\end{cases}
\]
The solution to this ODE is 
\[
k(t) = -\frac{2}{\sqrt \alpha} \arctan\left(\sqrt{\frac{\alpha}{2}} t\right)
\]
for $t \in (-\beta,\beta)$ where $\beta = \frac{\pi}{2}\sqrt{\frac{2}{\alpha}}$. Note that $k(t) \to \infty$ as $t\to -\beta$ and that $k(t)\to -\infty$ as $t\to \beta$. 

Define 
    \[
    \rho_0(x) = d_{\overline M}(x,\sigma), \quad x\in \overline M.
    \]
    Then let $\rho_1$ be a smooth approximation to $\rho_0$ which satisfies $\|\grad \rho_1(x)\| \le 2$ for all $x\in \overline M$, $\rho_1<1$ on $\sigma$, and $\rho_1>L+\beta+1$ on $\bd\Sigma_1=\Lambda$. Let $\rho_2$ be the restriction of $\rho_1$ to $\Sigma_1$. Now choose $\delta > 0$ sufficiently small so that $L-\beta+\delta $, $L+2\delta$, 
    and $L+\beta+ \delta$ are all regular values of $\rho_2$.  We can further ensure that $\Sigma_1 \cap B_{\overline M}(\sigma,L/4) \subset \{\rho_2 \le L-\beta -1\}$.  
   Now define 
    \[
    \rho(x) = \rho_2(x)-L-\delta, \quad x\in \Sigma_1.
    \]
    Then we have $\|\grad_{\Sigma_1} \rho(x)\| \le 2$ for all $x\in \Sigma_1$. We define 
\begin{gather*}
        \tilde{\Sigma}_1 = \{-\beta \le \rho \le \beta\}, \\
        \bd_+ \tilde{\Sigma}_1 = \{\rho = -\beta\},\\
        \bd_- \tilde{\Sigma}_1 = \{\rho = \beta\},\\
        \Omega_0 = \{-\beta \le \rho <  \delta\}. 
    \end{gather*}
    By construction, $\bd_{\pm}\tilde{\Sigma}_1$ and $\bd \Omega_0$ are all smooth hypersurfaces in $\tilde{\Sigma}_1$, and $\Omega_0$ contains $\bd_+ \tilde{\Sigma}_1$, and $\bd \Sigma_1 \cap \tilde{\Sigma}_1 = \emptyset$. Finally, we define 
    \[
    h(x) = k(\rho(x)).
    \]
    Then $h$ is a smooth function on the interior of $\tilde{\Sigma}_1$ which satisfies $h(x) \to \pm\infty$ as $x\to \bd_{\pm}\tilde{\Sigma}_1$. Moreover, we have 
    \[
\|\grad_{\Sigma_1}h(x)\| = \vert k'(\rho(x))\vert \cdot \|\grad_{\Sigma_1}\rho(x)\| \le 2\left(1 + \frac{\alpha}{2}k(\rho(x))^2\right) = 2 + \alpha h^2,
    \]
    for all $x\in \tilde{\Sigma}_1$. It follows that 
    \begin{equation}
    \label{eqn:new-h-inequality}
    2 + \alpha h^2 - \|\grad_{\Sigma_1} h\| \ge 0
    \end{equation}
    on $\tilde{\Sigma}_1$. Proposition \ref{prop: mu-bubble regularity} (with $u$ replaced by $u^a$) implies that there exists a minimizer $\Omega \subset \tilde{\Sigma}_1$ of the $\mu$-bubble functional 
    \[
\mathcal A(\Omega) = \int_{\bd \Omega} u^a\, d \mathcal H^{4} - \int_{\tilde{\Sigma}_1} (\chi_\Omega - \chi_{\Omega_0}) u^a h\, d\mathcal H^{5} 
\]
such that $\Omega \operatorname{\Delta} \Omega_0$ is compactly contained in the interior of $\tilde{\Sigma}_1$. Let $\Sigma_2$ be a component of $\bd \Omega$. 

According to Proposition \ref{proposition:variational-formulas} with $u$ replaced by $u^a$ (see also Mazet \cite{mazet2024stable} Section 4.2),  the first variation satisfies
\[
H_{\Sigma_2} = h - a u^{-1} \la \grad_{\Sigma_1} u, \nu\ra,
\] 
and the second variation formula implies 
\begin{align*}
0 &\le \int_{\Sigma_2} u^a \bigg[\vert \grad_{\Sigma_2} \psi\vert^2 - (\vert A_{\Sigma_2}\vert^2 + \ric_{\Sigma_1}(\nu,\nu))\psi^2 - a u^{-2} \la \grad_{\Sigma_1} u,\nu\ra^2 \psi^2 \\
&\qquad \qquad + a u^{-1} (\lap_{\Sigma_1} u - \lap_{\Sigma_2} u - H_{\Sigma_2} \la \grad_{\Sigma_1} u,\nu\ra)\psi^2 - \la \grad_{\Sigma_1} h,\nu\ra \psi^2\bigg]
\end{align*}
for all $\psi\in C^1(\Sigma_2)$. It follows that there exists a function $w > 0$ on $\Sigma_2$ which satisfies 
\begin{align*}
\div_{\Sigma_2}(u^a \grad_{\Sigma_2} w) &\le \bigg[ -(\vert A_{\Sigma_2}\vert^2 + \ric_{\Sigma_1}(\nu,\nu)) - a(\vert A_{\Sigma_1}\vert^2 + \ric_{\overline M}(\eta,\eta))\bigg] u^a w\\
&\qquad +\bigg[ - \la \grad_{\Sigma_1} h,\nu\ra - a u^{-2} \la \grad_{\Sigma_1} u,\nu\ra^2 - a H_{\Sigma_2} u^{-1} \la \grad_{\Sigma_1}u,\nu\ra - a \frac{\lap_{\Sigma_2} u}{u}\bigg] u^a w. 
\end{align*}
Now we turn our attention to the diameter bound. 

\begin{prop}
\label{n6m3diameter}
    The diameter of $\Sigma_2$ is uniformly bounded independently of $L$. 
\end{prop}

\begin{proof}
We are going to apply the Shen-Ye generalized Bonnet-Myers theorem with test function $u^a w$. Fix an orthonormal basis $\{e_1,e_2,e_3,e_4,e_5,e_6\}$ for $\overline M$ so that $e_1 = \eta$ and $e_2 = \nu$. By condition (iv), it suffices to  show that
\[
\ric_{\Sigma_2}(e_3,e_3) - \tau \frac{\lap_{\Sigma_2}(u^a w)}{u^a w} + \frac{\tau}{4} \vert \grad_{\Sigma_2} \ln(u^a w)\vert^2 \ge \kappa > 0
\]
for some $\kappa$ which does not depend on $L$. 
We have 
\begin{align*}
&\ric_{\Sigma_2}(e_3,e_3) - \tau \frac{\lap_{\Sigma_2}(u^a w)}{u^a w} + \frac{\tau}{4} \vert \grad_{\Sigma_2} \ln(u^a w)\vert^2 \\
&\qquad = \ric_{\Sigma_2}(e_3,e_3) - \tau \frac{\div_{\Sigma_2}(u^a \grad_{\Sigma_2}w)}{u^a w} - \tau \frac{\div_{\Sigma_2}(w \grad_{\Sigma_2}u^a)}{u^a w} + \frac{\tau}{4} \left\vert\frac{\grad_{\Sigma_2} u^a}{u^a} + \frac{\grad_{\Sigma_2} w}{w}\right\vert^2.
\end{align*}
Now observe that \begin{align*}
&- \tau \frac{\div_{\Sigma_2}(w \grad_{\Sigma_2}u^a)}{u^a w} + \frac{\tau}{4} \left\vert\frac{\grad_{\Sigma_2} u^a}{u^a} + \frac{\grad_{\Sigma_2} w}{w}\right\vert^2\\
&\qquad\quad  = -\tau \left[\left\la \frac{\grad_{\Sigma_2} u^a}{u^a}, \frac{\grad_{\Sigma_2} w}{w}\right\ra + a(a-1)\frac{\vert \grad_{\Sigma_2} u\vert^2}{u^2} + a \frac{\lap_{\Sigma_2} u}{u} \right] + \frac{\tau}{4} \left\vert\frac{\grad_{\Sigma_2} u^a}{u^a} + \frac{\grad_{\Sigma_2} w}{w}\right\vert^2\\
&\qquad\quad \ge -a\tau  \frac{\lap_{\Sigma_2} u}{u} -\tau \left\la \frac{\grad_{\Sigma_2} u^a}{u^a}, \frac{\grad_{\Sigma_2} w}{w}\right\ra + \frac{\tau}{4} \left\vert\frac{\grad_{\Sigma_2} u^a}{u^a} + \frac{\grad_{\Sigma_2} w}{w}\right\vert^2\\
&\qquad \quad = -a\tau \frac{\lap_{\Sigma_2} u}{u} + \frac{\tau}{4} \left\vert\frac{\grad_{\Sigma_2} u^a}{u^a} - \frac{\grad_{\Sigma_2} w}{w}\right\vert^2 \ge -a\tau  \frac{\lap_{\Sigma_2} u}{u},
\end{align*}
where we used the fact that $a < 1$. Thus we have 
\begin{align*}
&\ric_{\Sigma_2}(e_3,e_3) - \tau \frac{\lap_{\Sigma_2}(u^a w)}{u^a w} + \frac{\tau}{4} \vert \grad_{\Sigma_2} \ln(u^a w)\vert^2 \\
&\qquad \quad \ge \ric_{\Sigma_2}(e_3,e_3) + \tau \vert A_{\Sigma_2}\vert^2 + \tau \ric_{\Sigma_1}(e_2,e_2) + a \tau \vert A_{\Sigma_1}\vert^2 + a\tau \ric_{\overline M}(e_1,e_1) \\
&\qquad \qquad \qquad + \tau \la \grad_{\Sigma_1}h, e_2\ra + a \tau H_{\Sigma_2} u^{-1} \la \grad_{\Sigma_1} u,e_2\ra + a\tau u^{-2} \la \grad_{\Sigma_1} u,e_2\ra^2\\
&\qquad\quad =  \ric_{\Sigma_2}(e_3,e_3) + \ric_{\Sigma_1}(e_2,e_2) + \ric_{\overline M}(e_1,e_1) +  \mathcal R\\
&\qquad\qquad\qquad  + \tau \vert A_{\Sigma_2}\vert^2 + a \tau \vert A_{\Sigma_1}\vert^2 + (\tau-1) \sum_{q=3}^6 \left[A^{\Sigma_1}_{22} A^{\Sigma_1}_{qq} - (A^{\Sigma_1}_{2q})^2 \right]  
\\
&\qquad\qquad\qquad  +  \tau \la \grad_{\Sigma_1} h ,e_2\ra  + a\tau H_{\Sigma_2} u^{-1}\la \grad_{\Sigma_1}u,e_2\ra + a\tau u^{-2}\la \grad_{\Sigma_1} u,e_2\ra^2,
\end{align*}  where $\mathcal R$ denotes a sum of 9 terms of the form $(\tau - 1)R_{\overline M}(X,Y,X,Y)$ or $(a\tau - 1)R_{\overline M}(X,Y,X,Y)$ and therefore satisfies $\vert \mathcal R\vert \le \frac 1 2$ by condition (i) on $a$.

Define the quantity \begin{align*}
\mathcal K := &\ric_{\Sigma_2}(e_3,e_3) + \ric_{\Sigma_1}(e_2,e_2) + \ric_M(e_1,e_1) \\
&\quad  + \tau \vert A_{\Sigma_2}\vert^2 + a \tau \vert A_{\Sigma_1}\vert^2 + (\tau-1) \sum_{q=3}^6 \left[A^{\Sigma_1}_{22} A^{\Sigma_1}_{qq} - (A^{\Sigma_1}_{2q})^2 \right] 
\\
&\quad  +  \tau \la \grad_{\Sigma_1} h ,e_2\ra  + a\tau H_{\Sigma_2} u^{-1}\la \grad_{\Sigma_1}u,e_2\ra + a\tau u^{-2}\la \grad_{\Sigma_1} u,e_2\ra^2. 
\end{align*}
By the bound on $\mathcal R$, to prove the proposition, it suffices to show that $\mathcal K \ge 1$. According to \cite[Lemma 3.8]{brendle2024generalization}, we have 
\begin{align}
\label{eqn:K}
    \mathcal K = C_3(e_1,e_2,e_3) + \mathcal B_1 + \mathcal B_2 + \tau \la \grad_{\Sigma_1}h,e_2\ra,
\end{align}
where 
\[
\mathcal B_1 = a\tau \vert A_{\Sigma_1}\vert^2 + \tau \sum_{q=3}^6 \left[ A^{\Sigma_1}_{22}A^{\Sigma_1}_{qq}-(A^{\Sigma_1}_{2q})^2\right] + \sum_{q=4}^6 \left[A^{\Sigma_1}_{33}A^{\Sigma_1}_{qq}-(A^{\Sigma_1}_{3q})^2\right]
\]
and 
\[
\mathcal B_2 = \tau \vert A_{\Sigma_2}\vert^2 +\sum_{q=4}^6  \left[A^{\Sigma_2}_{33}A^{\Sigma_2}_{qq} - (A^{\Sigma_2}_{pq})^2\right] + a\tau H_{\Sigma_2} u^{-1} \la \grad_{\Sigma_1} u,e_2\ra + a\tau u^{-2}\la \grad_{\Sigma_1} u ,e_2\ra^2.
\]
Next we focus on obtaining lower bounds for $\mathcal B_1$ and $\mathcal B_2$.

\begin{lem}
    \label{lemB1}The quantity $\mathcal B_1$ is non-negative.
\end{lem}

\begin{proof}
Dropping the subscript and superscript $\Sigma_1$'s, we have 
\begin{align*}
\mathcal B_1 = a\tau \vert A\vert^2 + \tau \sum_{q=3}^6 \left[ A_{22}A_{qq}-A_{2q}^2\right] + \sum_{q=4}^6 \left[A_{33}A_{qq}-A_{3q}^2\right].
\end{align*} 
Since $\Sigma_1$ is minimal and $a\tau>a^2>1/2$,  it follows that 
\begin{align*}
\mathcal B_1 &\ge  a\tau (A_{22}^2 + A_{33}^2 + A_{44}^2 + A_{55}^2 + A_{66}^2) + \tau (A_{22}A_{33} + A_{22}A_{44}+A_{22}A_{55}+A_{22}A_{66}) \\
&\qquad  + (A_{33}A_{44} + A_{33}A_{55} + A_{33}A_{66})\\&= (a\tau - \frac 1 2)(A_{22}^2 + A_{33}^2 + A_{44}^2 + A_{55}^2 + A_{66}^2) + \frac 1 2 (A_{22}+A_{33}+A_{44}+A_{55}+A_{66})^2 \\
&\qquad + (\tau-1)(A_{22}A_{33}+A_{22}A_{44}+A_{22}A_{55}+A_{22}A_{66}) -A_{44}A_{55}-A_{44}A_{66}-A_{55}A_{66}\\
&= (a\tau - \frac 1 2)(A_{22}^2 + A_{33}^2 + A_{44}^2 + A_{55}^2 + A_{66}^2)  \\
&\qquad + (\tau-1)(A_{22}A_{33}+A_{22}A_{44}+A_{22}A_{55}+A_{22}A_{66}) -A_{44}A_{55}-A_{44}A_{66}-A_{55}A_{66}.
\end{align*}
Again, since $\Sigma_1$ is minimal, we have 
\[
A_{22}^2 + A_{33}^2 \ge \frac{1}{2}(A_{22}+A_{33})^2 = \frac{1}{2}(A_{44} + A_{55} + A_{66})^2. 
\]
Thus we obtain 
\begin{align*}
\mathcal B_1 &\ge 2(1-\tau) (A_{22}^2+A_{33}^2+A_{44}^2+A_{55}^2+A_{66}^2) + (\tau-1)(A_{22}A_{33}+A_{22}A_{44}+A_{22}A_{55}+A_{22}A_{66}) \\
&\qquad + \frac 3 2 (2\tau - \frac 5 2 + a\tau)(A_{44}^2 + A_{55}^2 + A_{66}^2)  + (2\tau - \frac 7 2 + a\tau)(A_{44}A_{55}+A_{44}A_{66}+A_{55}A_{66})\\
&\ge \frac 3 2 (2\tau - \frac 5 2 + a\tau)(A_{44}^2 + A_{55}^2 + A_{66}^2)  + (2\tau - \frac 7 2 + a\tau)(A_{44}A_{55}+A_{44}A_{66}+A_{55}A_{66}).
\end{align*}
This will be non-negative as long as 
\[
\frac{3}{2}(2\tau - \frac 5 2 + a\tau) \ge \frac 7 2 - 2\tau - a\tau. 
\]
This is equivalent to 
\[
5\tau + \frac 5 2 \tau a \ge \frac{29}{4},
\]
which holds since $\tau > a$ and $a > \frac{\sqrt{390}}{10}-1$. 
\end{proof}

\begin{lem}
    \label{lemB2}The quantity $\mathcal B_2$ satisfies $\mathcal B_2 \ge \alpha h^2$.
\end{lem}

\begin{proof}
Using $H_{\Sigma_2} = h -au^{-1}\la \grad_{\Sigma_1} u,\eta\ra$, and dropping the subscript and superscript $\Sigma_2$'s, we have 
\begin{align*}
\mathcal B_2 - \alpha h^2 = \tau \vert A\vert^2 +\sum_{q=4}^6  \left[A_{33}A_{qq} - (A_{pq})^2\right] + \tau H(h-H) + \frac{\tau}{a} (H-h)^2 - \alpha h^2. 
\end{align*}
Now we analyze this quadratic form as in Mazet \cite{mazet2024stable}. Write 
\[
A = \begin{pmatrix}
\frac{H}{4} + \Phi_{33} & A_{34} & A_{35} & A_{36} \\
A_{43} & \frac{H}{4} + \Phi_{44} & A_{45} & A_{46}\\
A_{53} & A_{54} & \frac{H}{4} + \Phi_{55} & A_{56}\\
A_{63} & A_{64} & A_{65} & \frac H 4 + \Phi_{66}
\end{pmatrix}.
\]
Then, since $\tau \ge \frac 1 2$, we have 
\begin{align*}
&\tau \vert A\vert^2 +\sum_{q=4}^6  \left[A_{33}A_{qq} - (A_{pq})^2\right] + \tau H(h-H) + \frac{\tau}{a} (H-h)^2 - \alpha h^2\\
&\qquad \ge \tau \frac{H^2}{4} + \tau \vert \Phi\vert^2 + A_{33}(H-A_{33}) + \tau H (h-H) + \frac{\tau}{a}(H-h)^2 - \alpha h^2\\
& \qquad = \tau \frac{H^2}{4} + \tau \vert \Phi\vert^2 + (\frac{H}{4} + \Phi_{33})(\frac{3}{4}H - \Phi_{33}) + \tau H(h-H) + \frac{\tau}{a}(H-h)^2 - \alpha h^2\\
&\qquad = \left(\frac 3 {16} - \frac 3 4 \tau + \frac{\tau}{a}\right) H^2 + \left(\frac{\tau}{a}-\alpha\right) h^2 + \tau\vert \Phi\vert^2 - \Phi_{33}^2 + \left(\tau - \frac{2\tau}{a}\right)Hh + \frac{1}{2}H \Phi_{33}.
\end{align*} 
Let $E = \{(x_3,x_4,x_5,x_6) \in \R^4: x_3 + x_4 + x_5 + x_6 = 0\}$.  Then any vector in $E$ can be written in the form 
\[
\frac{1}{\sqrt 2}\begin{pmatrix} 
0 \\ 0 \\ 1 \\ -1 
\end{pmatrix} y_1 +
\frac{1}{\sqrt 6} \begin{pmatrix}
0 \\ 2 \\ -1 \\ -1 
\end{pmatrix}y_2 + 
\frac{1}{\sqrt{12}}\begin{pmatrix}
3 \\ -1 \\ -1 \\ -1
\end{pmatrix} z.
\]
Expressing $\Phi$ in these coordinates, the previous quadratic form is at least 
\[
\left(\frac 3 {16} -\frac 3 4 \tau + \frac{\tau}{a}\right) H^2 + \left(\frac{\tau}{a} - \alpha\right) h^2 + \left(\tau - \frac 3 4\right)z^2 + \left(\tau - \frac{2\tau}{a}\right)Hh +  \frac{\sqrt 3}{4} H z.
\]
Condition (iii) on $a$, $\tau$, and $\alpha$ ensures that this quadratic form in $H$, $h$, and $z$ is positive definite. The lemma follows. 
\end{proof}

Combining (\ref{eqn:K}) with Lemma \ref{lemB1}, Lemma \ref{lemB2},  the assumption on the intermediate curvature, and (\ref{eqn:new-h-inequality}) we now obtain 
\[
\mathcal K \ge 3 + \alpha h^2 - \tau \|\grad_{\Sigma_1}h\| \ge 1.
\]
Proposition \ref{n6m3diameter} follows. 
\end{proof}

Finally, we can finish the proof of Theorem \ref{Theorem:n6m3}. Since $\bd {\Sigma_2} = 0$ and $\diam({\Sigma_2})$ is uniformly bounded and $H_{4}(\overline M,\Z) = 0$, it follows from Proposition \ref{Proposition:uniformly-acyclic} that ${\Sigma_2}$ bounds within an $R$-neighborhood for some constant $R$ that depends only on $\overline M$. Applying this argument to each component of $\bd \Omega$, we see that $\bd \Omega$ bounds within its $R$-neighborhood in $\overline M$. Since $d(\bd \Omega, \sigma) \ge L/4$ and $\bd \Omega$ is linked with $\sigma$ by construction, this is a contradiction provided we select $L > 2R$. This completes the proof.

\section{Mapping Version and Classification} \label{sec: mapping and classification}

In this section we prove a mapping version  and a refinement of Theorem \ref{Theorem:Main-Theorem} into a positive result. We start with the lifting Lemma proven by Chodosh-Li-Liokumovich \cite{chodosh2023classifying}.

\begin{lem}[Chodosh-Li-Liokumovich]\label{lem: lifting}
    Suppose that $X,N$ are closed oriented manifolds and $f: N\to X$ has non-zero degree. Letting $\bar{X}$ denote the universal covering of $X$, there exists a connected cover $\hat{N}\to N$ and a lift $\hat{f}:\hat{N}\to \bar{X}$ such that $\hat{f}$ is proper and $\deg \hat{f}=\deg f$.
\end{lem}

\begin{rem}
    It can be checked that $\hat{f}$ is a Lipschitz map.
\end{rem}

\begin{corollary}[Mapping Version]
Suppose $M^n$ is a closed  manifold and that the universal cover $\overline M$ of $M$ satisfies 
\[
H_n(\overline M,\Z) = H_{n-1}(\overline M,\Z) = \hdots = H_{n-m+1}(\overline M,\Z) = 0. 
\]
Assume further $N^n$ is a closed manifold with non-zero degree map to $M^n$,
then $N$ does not admit a metric of positive $m$-intermediate curvature in any of the following cases: $n\in\{3,4,5\}$ and $m\in \{1,2,\dots,n-1\}$;  $n = 6$ and $m\in\{1,2,\textcolor{teal}{3},4\}$; $n$ is arbitrary and $m = 1$.
\end{corollary}
\begin{proof}
    Since $\hat{f}$ is proper with non-zero degree and since it is a Lipschitz map, all the diameter estimate arguments are true in the connected cover $\hat{N}$ given by Lemma \ref{lem: lifting}, the result then follows.
\end{proof}

We now summarize the filling estimates for $n=6,m=4$ in previous sections as follows.

\begin{theorem}[Filling Estimates]\label{thm: fillingestimate}
    Suppose $(N^6,g)$ is a closed manifold with positive $4$-intermediate curvature no less than $\frac{3}{2}$. Fix a connected Riemannian cover $\hat{N}$:

         Consider a closed embedded $4$-manifold $\hat{\Lambda}$ such that $[\hat{\Lambda}]=0\in H_4(\hat{N};\mathbb Z)$. Then there exists a $5$-chain $\hat{\Sigma}\subset B_L(\hat{\Lambda})$ and a 
        $4$-dimensional submanifold $\hat{\Gamma}$ such that
        \[
        \partial\hat{\Sigma}=\hat{\Lambda}-\hat{\Gamma}
        \]
        as chains.

        Furthermore, there are $4$-chains $\hat{K}_1,\cdots, \hat{K}_s$ with diameter bounded by $L$ and $3$-cycles $\{\hat{\Upsilon}_j^l\}$ where $j=1,2,\cdots, s$ and $l=1,2,\cdots, n(j)$ such that
        \[
         \hat{\Gamma}=\sum_{j=1}^s\hat{K}_j, \quad \partial\hat{K}_j=\sum_{l=1}^{n(j)}\hat{\Upsilon}_j^l,
       \]
       where $\Upsilon_j^l$ has diameter bounded by $L$, and the above equalities hold as chains.

       Finally, there is an integer $q$ and a function 
       \[
       u:\{(j,l):j=1,2,\cdots, s, l=1,\cdots, n(j)\}\to \{1,2,\cdots, q\}
       \]
       such that for any $t\in\{1,2,\cdots, q\}$,
       \[
       \diam\left(\cup_{(j,l)\in u^{-1}(t)}\hat{\Upsilon}_j^l\right)\leq L
       \]
       and 
       \[
       \sum_{u(j,l)=u^{-1}(r)}\hat{\Upsilon}_j^l=0
       \]
       as chains.
\end{theorem}

\begin{theorem}
    Suppose $N^n$ is a closed manifold which admits positive $m$-intermediate curvature. Further suppose that either 
\begin{enumerate}
        \item  $n=5$, $m=2$, and $\pi_2(N^5)=0$; or
        \item  $n=6$, $m=4$, and $\pi_2(N^6)=\pi_3(N^6)=\pi_4(N^6)=0$. 
    \end{enumerate}
    Then a finite covering of $N$ is homeomorphic to $S^n$ or connected sum of $S^{n-1}\times S^1$. 
 \end{theorem}

\begin{proof} 
     The idea is to prove the fundamental group of $N$ is virtually free, and then with the aid of Theorem 1.3 in Gadgil-Seshadri \cite{gadgil2009topology}, we obtain the desired result.

    \textbf{Case 1: $n=5, m=2$}. 

    By Ma \cite[Theorem 1.6]{ma2024urysohn}, $\pi_1(N^5)$ is virtually free. Thus there exists a finite connected cover $\hat{N}$ of $N$ such that $\pi_1(\hat{N})$ is free. Since $\pi_2(N)=0$, by Gadgil- Seshadri \cite[Theorem 1.3]{gadgil2009topology}, the proof then follows.

    \textbf{Case 2: $n=6,m=4$}. 
    Since $\pi_2(N)=\pi_3(N)=\pi_4(N)=0$, by the Hurewicz Theorem, the universal covering $\tilde{N}$ of $N$ has trivial $4$th homology group. Then as a direct Corollary of Theorem \ref{thm: fillingestimate} and Proposition \ref{Proposition:uniformly-acyclic}, there exists an $L=L(N,g)>0$ such that for a closed embedded $4$-submanifold in $\tilde{N}$ is null-homologous in its $L$-neighborhood. By Chodosh-Li-Liokumovich \cite[Proposition 8]{chodosh2023classifying}, for any point $p\in \tilde{N}$, each connected component of a level set of $d(p,\cdot)$ has diameter bounded by $20 L$. By Chosdosh-Li-Liokumovich \cite[Corollary 14]{chodosh2023classifying}, $\pi_1(N)$ is virtually free. Since $\pi_2(N)=\pi_3(N)=0$, by Gadgil- Seshadri \cite[Theorem 1.3]{gadgil2009topology}, the proof then follows.
\end{proof} 

    With Lemma \ref{lem: lifting}, we have the following mapping version classification.

\begin{corollary}
    Suppose $N^n$ is a closed manifold admits positive $m$-intermediate curvature, assume further there exists a non-zero degree map $f: N\to X$, where $X$ is a closed manifold satisfying 
    \begin{enumerate}
        \item when $n=5$, $m=2$, $\pi_2(X^5)=0$;
        \item when $n=6$, $m=4$, $\pi_2(X^6)=\pi_3(X^6)=\pi_4(X^6)=0$.
    \end{enumerate}
    Then a finite covering of $N$ is homeomorphic to $S^n$ or connected sum of $S^{n-1}\times S^1$.
\end{corollary}

\bibliographystyle{plain}
\bibliography{reference.bib}
\end{document}